\documentclass[reqno,12pt]{amsart}
\usepackage{amsmath,amssymb,epsfig,color,here,wasysym,stmaryrd,slashbox}

\numberwithin{equation}{section}

\newtheorem{theorem}{Theorem}[section]

\newtheorem{lemma}[theorem]{Lemma}
\newtheorem{proposition}[theorem]{Proposition}
\newtheorem{definition}[theorem]{Definition}
\newtheorem{corollary}[theorem]{Corollary}

\theoremstyle{remark}
\newtheorem{example}[theorem]{Example}

\newcommand{\di}{\displaystyle}

\newcommand{\vanish}[1]{}
\newcommand{\Sym}[1]{\operatorname{Sym}\left( #1\right)}

\begin{document}

\title{Counting partitions of a fixed genus}

\author[Robert Cori and G\'abor Hetyei]{Robert Cori \and G\'abor Hetyei}

\address{Labri, Universit\'e Bordeaux 1, 33405 Talence Cedex, France.
\hfill\break
WWW: \tt http://www.labri.fr/perso/cori/.}

\address{Department of Mathematics and Statistics,
  UNC-Charlotte, Charlotte NC 28223-0001.
WWW: \tt http://www.math.uncc.edu/\~{}ghetyei/.}

\date{\today}


\begin{abstract}
We show that, for any fixed genus $g$, the ordinary generating function
for the genus $g$ partitions of an $n$-element set into $k$ blocks is 
algebraic. The proof involves showing that each such partition may be
reduced in a unique way to a primitive partition and that the number of
primitive partitions of a given genus is finite. We illustrate our
method by finding the generating function for genus $2$ partitions,
after identifying all genus $2$ primitive partitions, using a computer-assisted
search. 
\end{abstract}

\maketitle

\section*{Introduction}

Many recent developments in discrete mathematics concern the use of topological
concepts in order to obtain combinatorial properties of finite
objects. The combinatorial definition of the genus of graphs and their
embeddings using maps and hypermaps, allowed Jackson and
Visentin~\cite{Jackson-Visentin1, Jackson-Visentin2} to
obtain enumerative results using the theory of characters in the algebra
of the symmetric group. Their  
approach inspires to check if other combinatorial 
objects with certain values for the genus have deep  combinatorial
properties. A genus may be defined for a  partition by considering a
partition as a permutation whose cycles consists of the blocks written
in increasing order, and by using the combinatorial definition that
exists for graph embeddings. Notice, the noncrossing ones are exactly
those that have genus $0$. The noncrossing partitions first defined by
Kreweras~\cite{Kreweras} have a lot of interesting properties for which
many authors obtained remarkable results. In her M.S.\ thesis Martha
Yip showed that genus $1$ partitions have also nice combinatorial
properties, for instance a lattice structure was obtained. 

 In this paper we are interested in enumerative results. It is well
known that the number $d_{n,k}$ of genus 0 partitions on $n$ elements
containing $k$ blocks are the Narayana numbers. Their  generating power series
$$D(x,y) = \sum_{n,k \geq 0} d_{n,k} x^ny^k$$
is algebraic and is given by a rational function of  $x,y$ and
$t=\sqrt{(1-x-xy)^2-4x^2y}$.  The same statement also holds for genus
$1$ partitions: an explicit formula to this effect was conjectured by 
M.\ Yip and proved in~\cite{Cori-Hetyei}. The main result of this paper
is Theorem~\ref{thm:main}, stating that for any integer
$g \geq 0$, the generating function counting all partitions of a fixed
genus by the size of the underlying set and the number of parts is
algebraic and is given by a rational function of  $x,y$ and
$t=\sqrt{(1-x-xy)^2-4x^2y}$. We also obtain an explicit formula for
$g=2$, albeit the proof of this formula relies on computer assistance.  

We arrive at our main result after several reduction steps. First, in
Section~\ref{sec:er}, we introduce two types of {\em elementary reductions}
that do not change the genus: the removal of a fixed point and the
contraction of $i$ and $i+1$ whenever $i+1$ immediately follows $i$ in
the same part. The resulting class of reduced partitions is slightly
smaller than the one we introduced in~\cite{Cori-Hetyei} when we counted
genus $1$ partitions. The formula connecting the generating functions of
all, respectively reduced partitions is also an easy substitution rule,
involving the generating function of noncrossing partitions.
The key idea behind our main result is the notion of {\em parallel edges} and
{\em primitive partitions}, introduced in Section~\ref{sec:primitive}. When a
permutation $\alpha$ sends $i$ into $j+1$ and $j$ into $i+1$, we call the
directed edges $(i,j+1)$ and $(j,i+1)$ parallel. These may be
represented by parallel line segments in a diagram, and merging the two
parts of a partition along such a pair of parallel edges does not change
the genus.
The existence of a parallel pair of edges in a partition $\alpha$ on
$\{1,\ldots,n\}$ is equivalent to the existence of a $2$-cycle in the
permutation $\alpha^{-1}\zeta_n$ used in the definition of the genus.
We call a partition primitive if it has no parallel edges. An
easy cycle counting argument shows that for any fixed genus there are
only finitely many primitive partitions.

Our paper is organized as follows:
In the first three  sections we give the notation used in the paper then we give precise description  on the way
to obtain a reduced partition from a general one. We also consider the case when
the partition has blocks of size 2, and show that we can associate to any partition one with blocks of
size 2 for which the reduction process translates in a more tractable way. It is clear that
these partitions are exactly the set of fixed point free involutions also called matchings.

Section~\ref{sec:pcount} is devoted to the proof of a formula that allows to compute the generating function of general partitions from that of involutions.

In Section ~\ref{sec:primitive} we introduce the semi-primitive
  and primitive partitions and prove that the number of primitive
  partitions of fixed genus is finite. Hence their generating function
  is a polynomial.

In Section~\ref{sec:primred}
we show a way to count all reduced partitions that yield the same
primitive partition, after the removal of parallel edges. We obtain that
the generating function of all reduced partitions of a fixed genus is a
rational function, whose denominator is a power of $1-x^2y$. Combining
this result with the main theorem  the main
result of our paper follows immediately.

The last section containing general results on partitions of a fixed arbitrary
genus is Section~\ref{sec:extract}. Here we show how to obtain a formula
for the number of all partitions of a fixed genus, assuming we found the
generating function for the number of reduced partitions, after finding
all primitive partitions, using the results of
Section~\ref{sec:primred}. Our proposed method has two key ingredients:
first we rewrite the generating function of all partitions of a fixed
genus as a linear combination of Laurent polynomials multiplied by
negative powers of $\sqrt{(x+xy-1)^2-4x^2y}$, then we may use a formula
published by Gessel~\cite{Gessel} to extract the coefficients of the
monomials in the negative powers of $\sqrt{(x+xy-1)^2-4x^2y}$.

The remaining sections illustrate the power and limitations of our
proposed method. In Section~\ref{sec:gen1} we show that there are only
two primitive partitions of genus $1$, which may be easily found ``by
hand''. Computing the generating function counting all partitions of
genus $1$ takes only a few lines, quickly reproducing the main result
of~\cite{Cori-Hetyei}. In Section~\ref{sec:gen2} we see that the
situation for partitions of genus $2$ is already much more
complicated. There are $3032$ primitive partitions of genus $2$, which we
were able to find after proving that they all belong to one of four
types and relied on a computer on finding all partitions of each given
type. While computing the generating function of reduced partitions, one
of these types gives rise to finitely many {\em semiprimitive} partitions (see
Sections~\ref{sec:primitive} and ~\ref{sec:primred}) which we
counted manually. Implementing the method outlined in
Section~\ref{sec:extract} could be performed by hand, but the length of
the calculation would stretch the limits of a publication of reasonable
length. We relied on the assistance of Maple to find an explicit formula
for the generating function of all partitions of genus $2$. Extracting
the coefficients using Gessel's formula can be performed automatically by hand,
but the resulting formula is somewhat lengthy. The numbers we
obtain from our formulas agree with the numbers obtained by directly
counting genus $2$ partitions of up to $12$ elements, using another
computer program.

Our work raises several exciting questions. Is there a way to describe
the structure of the primitive partitions of a fixed genus, and is there
a way to count them without using a computer-assisted search? While counting
genus $2$ partitions we mostly relied on bounding the number of cycles
whose length is greater than $2$ to facilitate the work of our
computer. Perhaps there are more ways to reduce even the primitive
partitions even further. Another line of future research could involve
the study of types of generating functions that arise: these are all
linear combinations of Laurent polynomials, multiplied by a single
negative power of $\sqrt{(x+xy-1)^2-4x^2y}$. A list of some areas where
negative powers of $\sqrt{(x+xy-1)^2-4x^2y}$ occur in generating
functions was already given in~\cite{Cori-Hetyei}: these range from
counting faces in root polytopes~\cite{Hetyei-L,Simion} through counting
convex polyominoes~\cite{Gessel} to counting Jacobi
configurations~\cite{Strehl}. Finding a more compact formula for even
genus $2$ partitions depends on an ability to manipulate such
expressions more nimbly.

\section{Preliminaries}

The study of hypermaps goes back to the sixties, they serve as a tool to
encode a topological representations of a hypergraph, embedded in a
surface.  A hypermap is a pair of permutations $(\sigma,\alpha)$ on a set of
points  $\{1,2, \ldots, n\}$, such that $\sigma$ and $\alpha$ generate a
transitive subgroup of the symmetric group.

The {\em genus} $g(\sigma,\alpha)$  of a hypermap $(\sigma,\alpha)$ is
defined by the equation 
\begin{equation}
\label{eq:genusdef}
n + 2 -2g(\sigma,\alpha) = z(\sigma) + z(\alpha) + z(\alpha^{-1}
\sigma),
\end{equation}
where $z(\alpha)$ denotes the number of cycles of the permutation
$\alpha$. The genus is always a nonnegative integer, see~\cite{Jacques}.

An important special type of a hypermap is when one of the permutations
constituting it is 
$$
\zeta_n=(1,2,\ldots,n).
$$
Note that the cyclic subgroup generated by $\zeta_n$ is already
transitive, the same holds even more for the subgroup generated by
$\zeta_n$ and $\alpha$, for any permutation $\alpha$.

\begin{definition}
\label{def:permgenus}
The genus $g(\alpha)$ of a permutation $\alpha$ on the set $\{1,\ldots,n\}$
is the genus of the hypermap $(\zeta_n,\alpha)$. 
\end{definition}
As a direct consequence of (\ref{eq:genusdef}), the genus of the
permutation $\alpha$ is given by
\begin{equation}
\label{eq:permgenus}  
n + 1 -2g(\alpha) = z(\alpha) + z(\alpha^{-1} \zeta_n).
\end{equation}
Hypermaps of the form $(\zeta_n, \alpha)$ are often called
hypermonopoles (for instance in \cite{CautisJackson} or \cite{Yip}). 
To count permutations of a fixed genus, a general machinery was built by
S.~Cautis and D.~M.~Jackson~\cite{CautisJackson} and explicit
formulas were given by A.~Goupil and
G.~Schaeffer~\cite{Goupil-Schaeffer}. 

The subject of this paper is counting {\em partitions} of a fixed
genus. We call a permutation $\alpha$ a partition if in each of its
cycles, listing the elements beginning with the least element of the
cycle results in an increasing list of integers. Note that partitions of
the set $\{1,\ldots,n\}$, as defined above, are in an obvious bijection
with the the set-partitions of $\{1,\ldots,n\}$: we identify the set-partition
$P=(P_i)_{i=1,\ldots,k}$ of the set $\{1,2,\ldots n\}$ with the
permutation $\alpha_P$ having $k$ cycles, whose cycles are lists of the
elements of the blocks $P_i$, the elements are listed in increasing
order within each block. The genus of a partition $P$ is the genus of
the permutation $\alpha_P$.  

While counting partitions of a fixed genus, we will use the fact that
cyclic renumbering of the points takes partitions into partitions and it
preserves the genus.
\begin{proposition}
\label{prop:cyc}
If $\alpha$ is a partition of $\{1,\ldots,n\}$ of genus $g$, then the
same holds for $\zeta_n\alpha\zeta_n^{-1}$.  
\end{proposition}  

In Section~\ref{sec:gen2} we will use the following variant of
\cite[Lemma~5]{Chapuy}, shown in~\cite[Lemma 1.5]{Cori-Hetyei}.
An element $i$ of ${1,2,\ldots ,n}$ is a {\em  back point} of the permutation
$\alpha$ if $\alpha(i) < i$ and $\alpha(i)$ is not the smallest element in
its cycle (i.e., there exists $k$ such that $\alpha^k(i) <\alpha (i)$).
\begin{lemma}
\label{lem:backp}
For any  permutation $\alpha\in \Sym{n}$, 
the sum of the number of back points of the permutation $\alpha$ and the
number of those of $\alpha^{-1} \zeta_n$ is equal
to $2g(\alpha)$. 
\end{lemma}
In particular, since a partition has no back points,
Lemma~\ref{lem:backp} has the following consequence.
\begin{corollary}
\label{cor:backp}
If the partition $\alpha$ of $\{1,\ldots,n\}$ has genus $g$ then the
permutation $\alpha^{-1} \zeta_n$ has $2g$ back points.  
\end{corollary}  

It was shown in  \cite[Theorem~1]{Cori} that a partitions of 
genus $0$ are exactly the noncrossing partitions.
A formula for the number of partitions of genus $1$ of a given number of
elements with a given number of parts was conjectured by
M.\ Yip~\cite{Yip}, and shown by the present
authors~\cite{Cori-Hetyei}.

\section{Elementary reductions and reduced permutations}
\label{sec:er}

In this section we will describe a reduction process that is similar to
the one we used in~\cite{Cori-Hetyei} to prove the formula for the
partitions of genus one, but the class of reductions used will be
slightly larger this time. Here and in the rest of the paper we will
often write $i+j$ instead of $\zeta_n^{j}(i)$, that is, all additions and
subtractions are understood to be taken modulo the number of elements in
the underlying set. 
\begin{figure}[h]
\begin{center}
\input{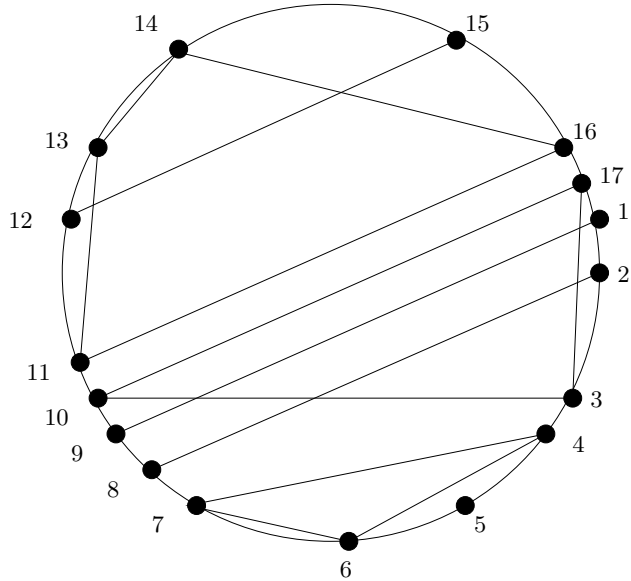}
\end{center}
\caption{The general partition $\alpha_0$ of genus $2$}
\label{fig:general}
\end{figure}

Consider the partition
$$
\alpha_0=(1,9)(2,8)(3,10,17)(4,6,7)(5)(11,13,14,16)(12,15)
$$
This partition on $17$ elements
satisfies
\begin{equation}
\label{eq:a0}  
\alpha_0^{-1}\zeta_{17}=(1,8)(2,17,9,3,7)(4,5)(6)(10,16)(11,15,14,12)(13) 
\end{equation}
and it has genus $2$ by (\ref{eq:permgenus}). The partition $\alpha_0$
is shown in Figure~\ref{fig:general}. We may represent every permutation
$\alpha$ in a similar way: we put the elements $1,\ldots,n$ on a circle in
clockwise order and for each $i$ put a directed edge beginning at $i$
and ending at $\alpha(i)$. We may always omit indicating the loops
corresponding to the fixed points of $\alpha$, and for the two-cycles we
may always use a single undirected edge to replace a pair of directed
edges between the same pair of points. Finally, for partitions we may
also omit indicating the orientation on the remaining edges, as each cycle of
length at least $3$ of a permutation that is a partition is oriented
clockwise. We call the resulting figure the {\em diagram of the
  partition}. 

We will repeatedly use reduction steps of two kinds: one removes a fixed
point of $\alpha$, and the other one removes a fixed point of
$\alpha^{-1}\zeta_n$. We define these steps for permutations in general,
and the main result in this section will be about classes of
permutations that do not have to be partitions. 

\begin{definition}
\label{def:er1}
Let $\alpha$ be any permutation of $\{1,\ldots,n\}$.   
An {\em elementary reduction of the first kind} is the removal of a
fixed point of $\alpha$:  given an $i$ such that $\alpha(i)=i$, we
remove the cycle $(i)$ from the cycle decomposition of $\alpha$ and
replace all $j>i$ by $j-1$, thus obtaining the cycle decomposition of a
permutation $\alpha'$ of $\{1,\ldots,n-1\}$. We call the inverse of this
operation, assigning $\alpha$ to $\alpha'$, an {\em elementary
  extension of the first kind}.     
\end{definition}  
For example, the permutation $\alpha_0$ given in (\ref{eq:a0}) (shown in
Figure~\ref{fig:general}) has $5$ as a fixed point. Removing this fixed
point is an elementary 
reduction of the first kind, giving rise to the permutation
\begin{equation}
\label{eq:ar1}  
(1,8)(2,7)(3,9,16)(4,5,6)(10,12,13,15)(11,14).
\end{equation}
Removing a fixed point in such a manner does not change the genus of a
permutation, because of the following two, obvious observations.
\begin{lemma}
\label{lem:fixp}  
Let $\alpha$ be a permutation of $\{1,\ldots,n\}$. Then $i$ is a fixed
point of $\alpha$ if and only if $\alpha^{-1}\zeta_n$ takes $i-1$ into
$i$. 
\end{lemma}
\begin{lemma}
Let $\alpha$ be a permutation of $\{1,\ldots,n\}$, satisfying
$\alpha(i)=i$ and let $\alpha'$ be the permutation of $\{1,\ldots,n-1\}$
obtained from $\alpha$ by removing the fixed point $i$. Then the cycle
decomposition of $\alpha'^{-1}\zeta_{n-1}$ is obtained from the cycle
decomposition of $\alpha^{-1}\zeta_n$ by removing $i$ from the cycle
containing it and decreasing all $j>i$ by one. 
\end{lemma}  
As a consequence of these two lemmas, the permutation $\alpha'$,
obtained from $\alpha$ the removal of the fixed point $i$ satisfies
$$
z(\alpha')=z(\alpha)-1 \quad\mbox{and}\quad
z(\alpha'^{-1}\zeta_{n-1})=z(\alpha^{-1}\zeta_n).   
$$
The number of permuted elements decreased by one, so the genus remains
unchanged by (\ref{eq:permgenus}). 
The second elementary reduction we will use is completely analogous: it
removes a fixed point of $\alpha^{-1}\zeta_n$. Note that, in analogy to
Lemma~\ref{lem:fixp}, the following observation holds.

\begin{lemma}
  \label{lem:dfixp}
Let $\alpha$ be a permutation of $\{1,\ldots,n\}$. Then $i$ is a fixed
point of $\alpha^{-1}\zeta_n$ if and only if $\alpha(i)=i+1$.
\end{lemma}  
Keeping in mind Lemma~\ref{lem:dfixp}, we call an element $i$
satisfying $\alpha(i)=i+1$ a {\em dual fixed point} of $\alpha$.
\begin{definition}
\label{def:er2}
Let $\alpha$ be any permutation of $\{1,\ldots,n\}$.   
An {\em elementary reduction of the second kind} is the removal of a
dual fixed point of $\alpha$ as follows:  given an $i<n$ such that
$\alpha(i)=i+1$, in the decomposition of $\alpha$ we
remove $i$ from the cycle containing it, and replace all $j>i$ by
$j-1$, thus obtaining the cycle decomposition of a permutation $\alpha'$
of $\{1,\ldots,n-1\}$. We call the inverse of this
operation, assigning $\alpha$ to $\alpha'$, an {\em elementary
  extension of the second kind}.     
\end{definition}  

Just like an elementary reduction of the first kind, an elementary
reduction of the second kind does not change the genus. 
For the permutation $\alpha_0$ given in (\ref{eq:a0}), the permutation
$\alpha_0^{-1}\zeta_{17}$ has two fixed points: we have $\alpha_0(6)=7$
and $\alpha_0(13)=14$. After removing the fixed point $5$, for 
resulting permutation $\alpha_1$ not only $5$ and $12$ are dual fixed points
of $\alpha$, but we also have $\alpha_1(4)=5$. We may continue removing
fixed points and dual fixed points, until we end up with a permutation
having no fixed points nor dual fixed points. Note that in some cases we
may end up removing all permuted elements, and so we will consider the
degenerate bijection of the empty set with itself a permutation. This
case does not arise if we start with a permutation that has positive genus.

\begin{definition}
We call a permutation $\alpha$ of $\{1,\ldots,n\}$ {\em reduced} if
neither $\alpha$ nor $\alpha^{-1}\zeta_n$ has any fixed points. 
\end{definition}  

In the case of $\alpha_0$ given in (\ref{eq:a0}) we may perform
elementary reductions in many ways, but after removing all fixed points
and dual fixed points, we will always end up with the
same reduced permutation $\alpha_1$, shown in Figure~\ref{fig:reduced}. 

\begin{figure}[h]
\begin{center}
\input{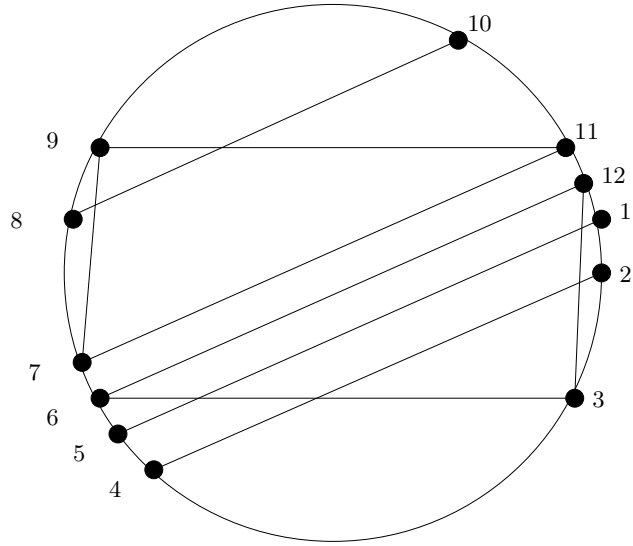}
\end{center}
\caption{The reduced partition $\alpha_1$ obtained from $\alpha_0$}
\label{fig:reduced}
\end{figure}

It is a trivial fact that each elementary reduction steps takes a
partition into a partition. Furthermore, even though there may be many
ways to perform the elementary reductions, if we keep performing these
operations until no such operation can be performed, the resulting
reduced permutation is unique.

\begin{theorem}
\label{thm:ured}
  Let $\alpha$ be any permutation on $\{1,\ldots,n\}$. Let us keep
  performing elementary reductions until we arrive at a reduced
  permutation. The resulting permutation is unique, regardless of the
  order in which the elementary reductions were performed.
\end{theorem}
\begin                            {proof}
Let us call an element $i\in \{1,2,\ldots,n\}$ {\em removable} if all
elements $j$ belonging to the clockwise arc $[i,\alpha(i)]$ have the
property that $\alpha(j)$ also belongs to the same arc and that for any
two elements $j,j'$ on the clockwise arc $[i,\alpha(i)]$, the directed
edges $j\rightarrow \alpha(j)$ and $j'\rightarrow \alpha(j')$  do not cross:
equivalently, the clockwise arcs $[j,\alpha(j)]$ and $[j,\alpha(j')]$ are
either disjoint, or they contain each other, or their intersection is a
subset of $\{j,j',\alpha(j),\alpha(j')\}$. It is easy to see by
  induction on the number of elements contained in the clockwise arc
  $[i,\alpha(i)]$ that an element $i$ is removable, if and only of it
  gets removed during any sequence of elementary reduction steps that
  does not stop before reaching a reduced permutation. By the end of any
  such reduction process, each number $i$ that is not removable is
  decreased by the number of removable elements in the interval $[1,i]$.  
\end{proof}

Due to Theorem~\ref{thm:ured}, any  class of permutations that is closed
under elementary reductions and extensions is completely determined by
the reduced permutations in the class. Knowing how to count the reduced
permutations should enable us to count all permutations in the
class. We will provide a formula for this in Section~\ref{sec:pcount}. 

\section{Bicolored matchings and reduced permutations}
\label{sec:bic}

The elementary reductions of the first and second kind, introduced in
Section~\ref{sec:er}, are duals of each other. In this section we review
a representation of permutations as bicolored matchings: in this setting
the duality becomes more apparent, making the proofs in
Section~\ref{sec:pcount} easier to present. The bicolored construction presented
here is a simplification of the four-colored construction presented
in~\cite{Cori}.   

Let $\alpha$ be a permutation of the set $\{1,\ldots,n\}$. We define the
{\em bicolored matching} $\mu[\alpha]$ associated to $\alpha$ as
the following permutation of the set $\{\pm 1,\ldots, \pm n\}$: for each
positive $i$, the permutation $\mu[\alpha]$ matches $i$ with
$-\alpha(i)$ and they form the $2$-cycle
$(i,-\alpha(i))$. Equivalently, for each $i>0$, the 
permutation $\mu[\alpha]$ contains the $2$-cycle $(-i,
\alpha^{-1}(i))$. The permutation $\mu[\alpha]$ is a fixed point
free involution arising as a product of $n$ $2$-cycles which we may
think of as edges in a matching of the set $\{\pm 1,\ldots,\pm n\}$,
where each edge matches a positive element with a negative element.   
\begin{example}
\label{ex:1}
For $\alpha=(1,5,3,4,8)(2,7)(6)$ we have\\
$\mu[\alpha]=(1,-5)(5,-3)(3,-4)(4,-8)(8,-1)(2,-7)(7,-2)(6,-6)$. 
\end{example}
The following statement is a degenerate case of \cite[Ch 1,
  Propri\'et\'e II.1]{Cori}, regarding the representability of a
hypermap as a map. 
\begin{lemma}
Let $\alpha$ be a permutation of $\{1,\ldots,n\}$ and let
$\zeta_n=(1,2,\ldots,n)$. Then the genus of the hypermap
$(\zeta_n,\alpha)$ is the same as the genus of the map
$(\widetilde{\zeta}_{2n},\mu[\alpha])$ where
$\widetilde{\zeta}_{2n}=(-1,1,-2,2,\ldots,-n,n)$. 
\end{lemma}
We may represent the matching $\mu[\alpha]$ by marking the
elements $-1,1,-2,2,\ldots,-n,n$ on a circle in this order and then
connecting the elements $i$ and $\mu[\alpha](i)$ with an edge, as shown
in Figure~\ref{fig:bic}. 
\begin{figure}[h]
\begin{center}
\input{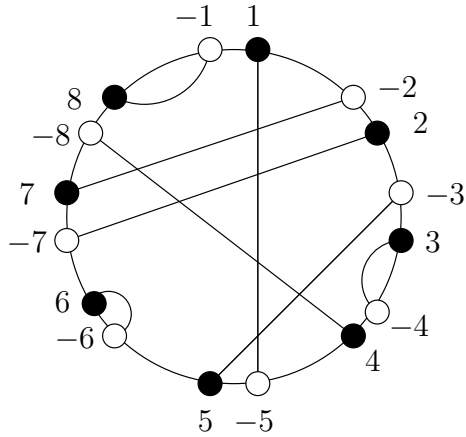}
\end{center}
\caption{Bicolored matching associated to $\alpha=(1,5,3,4,8)(2,7)(6)$}
\label{fig:bic}
\end{figure}

The effect of an elementary reduction of the first kind on the
associated bicolored matchings is the following.

\begin{lemma}
\label{lem:er1}
Let $\alpha$ be a permutation on $\{1,\ldots,n\}$ satisfying
$\alpha(i)=i$, end let $\alpha'$ be the permutation on
$\{1,\ldots,n-1\}$ obtained from $\alpha$ by removing the fixed point
$i$. Then $(-i,i)$ is an edge of $\mu[\alpha]$ and $\mu[\alpha']$ is
obtained from $\mu[\alpha]$ by removing the edge $(i,-i)$ and
decreasing the absolute value of the labels on the points
$$-(i+1), (i+1), -(i+2), i+2, \ldots, -n, n$$
by one (without changing their sign).
\end{lemma}  

To describe the similar effect of an elementary reduction of the second kind
on the associated bicolored matchings we need to distinguish two cases,
depending on whether the dual fixed point we want to remove is $n$ or a
point with a smaller label. 

\begin{lemma}
\label{lem:er2a}  
Let $\alpha$ be a permutation on $\{1,\ldots,n\}$ satisfying
$\alpha(i)=i+1$ for some $i<n$, end let $\alpha'$ be the permutation on
$\{1,\ldots,n-1\}$ obtained from $\alpha$ by removing the dual fixed point
$i$. Then $(i,-(i+1))$ is an edge of $\mu[\alpha]$ and $\mu[\alpha']$ is
obtained from $\mu[\alpha]$ by removing the edge $(i,-(i+1))$ and
and decreasing the absolute value of the labels on the points
$$i+1, -(i+2), i+2, \ldots, -n, n$$
by one (without changing their sign). 
\end{lemma} 

\begin{lemma}
\label{lem:er2b}  
Let $\alpha$ be a permutation on $\{1,\ldots,n\}$ satisfying
$\alpha(n)=1$ for some $i<n$, end let $\alpha'$ be the permutation on
$\{1,\ldots,n-1\}$ obtained from $\alpha$ by removing the dual fixed point
$n$. Then $(n,-1)$ is an edge of $\mu[\alpha]$ and $\mu[\alpha']$ is
obtained from $\mu[\alpha]$ by removing the edge $(n,-1)$ and
and changing the label $-n$ to $-1$.  
\end{lemma} 

The proof of Lemmas~\ref{lem:er1}, \ref{lem:er2a} and \ref{lem:er2b} is
left to the reader. They may be summarized as follows. Fixed points and
dual fixed points of $\alpha$ correspond to consecutive
elements of the cycle $(-1,1,-2,2,\ldots,-n,n)$ being matched
in $\mu[\alpha]$. An elementary reduction induces removing such a
consecutive pair, and adjusting the remaining labels in such a way that 
we obtain the cycle $(-1,1,-2,2,\ldots,-(n-1),n-1)$. 

\begin{corollary}
A permutation $\alpha$ of $\{1,\ldots,n\}$ is reduced if and only if the
corresponding bicolored matching $\mu[\alpha]$ does not match any two
consecutive elements on the cycle $(-1,1,-2,2,\ldots,-n,n)$.
\end{corollary}
Thus the set of all reduced permutations of $\{1,\ldots,n\}$ is
bijectively equivalent to the set of all systems of interlacing chords
joining $2n$ points on circle. These are counted by \cite[sequence
  A003436]{OEIS}.

It is also easy to describe in terms of the associated bicolored
matchings, which edges get removed in the reduction process that yields
the unique reduced permutation which exists according to
Theorem~\ref{thm:ured}.

\begin{definition}
\label{def:removable}  
Let $\alpha$ be a permutation of $\{1,\ldots,n\}$ and let $\mu[\alpha]$
be the associated bicolored matching. We call an edge 
$(u,v)$ in $\mu[\alpha]$ {\em removable} if it has the following
properties:
\begin{enumerate}
\item No other edge crosses the edge $(u,v)$.
\item Consider the sublists $u,\ldots, v$ and $v,\ldots, u$ of the
  circular order $(-1,1,\ldots,-n,n)$. Then, for at least one of these
  sublists, the set of edges connecting the elements of this sublist form
  a set of pairwise non-crossing edges.
\end{enumerate}
\end{definition}  

Let $\alpha$ be a permutation of $\{1,\ldots,n\}$ and let $\alpha'$ be
the unique reduced permutation $\alpha'$ of $\{1,\ldots,n'\}$ that may
be obtained by performing a sequence of elementary reductions, starting
with $\alpha$. Starting with $\alpha$, let us perform a sequence of
elementary reductions yielding $\alpha'$, and let us simultaneously
remove the appropriate edges and relabel the remaining points as
indicated in Lemmas~\ref{lem:er1} through \ref{lem:er2b} above.

\begin{proposition}
\label{prop:ured}
The set of edges removed from $\mu[\alpha]$ in the process described
above is exactly the set of removable edges. 
\end{proposition}
\begin{proof}
If $(u,v)$ gets removed in any elementary reduction process leading to
the bicolored matching $\mu[\alpha']$ then $(u,v)$ must be a removable
  edge. Indeed, an elementary reduction allows only removing an edge
  connecting consecutive elements of the circular list
  $(-1,1,\ldots,-n,n)$. If, after a certain number of elementary
  reductions, the image of $v$ consecutively follows the image of $u$ in
  the circular order then all edges having at last one endpoint 
from the sublist $u,\ldots,v$ must be formerly reduced edges, and thus
have both endpoints belong to the list and form a set of pairwise
non-crossing edges. Similarly, if after a certain number of
elementary reductions, the image of $u$ consecutively follows the image
of $v$ in the circular order then the sublist $v,\ldots,u$ of the
original circular order must support a set of pairwise non-crossing
edges. 

Conversely, in any sequence of elementary reductions that leads to a
reduced permutation, every removable edge will be removed at some point.
This may be shown by induction on the number of pairwise non-crossing
edges nested by the edge $(u,v)$, forming the set of all edges connecting
the endpoints on the sublist $u,\ldots,v$ or $v,\ldots,u$. From here on
we distinguish two cases:

{\bf \noindent Case 1:} The permutation $\alpha$ has genus zero. In this
degenerate case no two edges of $\mu[\alpha]$ cross and $\alpha'$ is the
empty permutation. 
 
{\bf \noindent Case 2:} The permutation $\alpha$ has positive genus.
In this case, for any removable edge $(u,v)$, exactly one of the sublists
$u,\ldots, v$ and $v,\ldots, u$ of the circular order
$(-1,1,\ldots,-n,n)$ has the non-crossing property described in item (2)
of Definition~\ref{def:removable}, otherwise $\mu[\alpha]$ would consist of
pairwise non-crossing edges and have genus zero. Without loss of
generality, we may assume that $(u,\ldots, v)$ is the sublist with the
noncrossing property. It is easy to see that all edges connecting two
points on this list are also removable. We may now show that all
removable edges get removed by induction on the number of elements in
the sublist $(u,\ldots,v)$. 
\end{proof}

\section{Counting permutations whose set is closed under elementary reductions}
\label{sec:pcount}

Due to Theorem~\ref{thm:ured}, any  class of permutations that is closed
under elementary reductions and extensions is completely determined by
the reduced permutations in the class. Knowing how to count the reduced
permutations should enable us to count all permutations in the
class. The main result of this section contains a formula telling how to
do this when we are interested in the number of permutations on a given
number of elements with a given number of cycles. To state our main
result we will need to use the generating function
$$
D(x,y)=1+\sum_{n=1}^{\infty} \sum_{k=1}^n \frac{1}{n}
\binom{n}{k}\binom{n}{k-1} x^ny^k
$$
of noncrossing partitions. As it is well-known \cite[sequence
  A001263]{OEIS}, the coefficient of $x^ny^k$ in  
$D(x,y)$ is the  number of non-crossing partitions of the set
$\{1,\ldots,n\}$ having $k$ parts. Note that we deviate from the usual
conventions by defining the constant term to be $1$, i.e. we consider
that there is one non-crossing partition on the empty set and it has
zero blocks. This generating function is given by the formula
\begin{equation}
\label{eq:dxy}
D(x,y)=\frac{1-x-xy-\sqrt{(x+xy-1)^2-4x^2y}}{2\cdot x}+1
\end{equation}
In the proof of the main result of this section we will use the
following lemma.
\begin{lemma}
\label{lem:ncbic}
The bicolored matching $\mu[\alpha]$ on the set $\{\pm 1, \ldots, \pm
n\}$ corresponds to a noncrossing partition if and only of the set of
edges of $\mu[\alpha]$ is noncrossing. Furthermore the number of cycles
of $\alpha$ equals the number of $2$-cycles $(u,-v)$ in
$\mu[\alpha]$ satisfying $1\leq v\leq u\leq n$.
\end{lemma}
\begin{proof}
  Consider a cycle $(i_1,\ldots,i_m)$ of $\alpha$, where $i_1$ is the
  least element of the cycle. This cycle corresponds to a set of edges
  $\{(i_1,-i_2),(i_2,-i_3),\ldots,(i_{m-1}-i_m), (i_m,-i_1)\}$. 
  These edges are pairwise noncrossing if and only if $i_1< \cdots
  < i_k$
  holds. We obtain that $\alpha$ is a partition if and only if edges of
  $\mu[\alpha]$ associated to the same cycle are pairwise
  noncrossing. In this case, for each cycle $(i_1,\ldots,i_m)$, the edge
  $(i_m,-i_1)$ pointing from the positive copy of the largest element to
  the negative copy of the smallest element in the cycle of $\alpha$ is
  the only edge between a positive point with larger and a negative
  point with smaller absolute value. Finally, for a partition $\alpha$
  is noncrossing if any pair of edges of $\mu[\alpha]$, associated to
  different cycles in $\alpha$, is also noncrossing.   
\end{proof}

The main result of this section is the following. 

\begin{theorem}
\label{thm:fromreduced}
Consider a class ${\mathcal C}$ of permutations that is closed under 
elementary reductions and extensions. Let $p(n,k)$ and $r(n,k)$
respectively be the number of all, respectively all reduced permutations
of $\{1,\ldots,n\}$ in the class having $k$ cycles. Then the generating
functions $P(x,y):=\sum_{n,k} p(n,k) x^ny^k$ and
$R(x,y):=\sum_{n,k} r(n,k) x^ny^k$ satisfy the equation
$$
P(x,y)=R\left(\frac{D(x,y)-1}{y},y\right) \cdot 
 \frac{1}{\sqrt{(x+xy-1)^2-4x^2y}}
$$
\end{theorem}
We will prove Theorem~\ref{thm:fromreduced} in two stages. First we show the
following result.
\begin{theorem}
\label{thm:red2}
Let ${\mathcal C}$ be a class of permutations that is closed under 
elementary reductions and extensions and let $P(x,y)$ and $R(x,y)$ be
the generating functions defined in Theorem~\ref{thm:fromreduced}. Then we
have 
$$
P(x,y)=R\left(\frac{xD(x,y)(D(x,y)+y-1)}{y},y\right)
\left(1+\frac{x \frac{\partial}{\partial x} D(x,y)}{D(x,y)}+\frac{x
\frac{\partial}{\partial x} D(x,y)}{D(x,y)+y-1}\right).
$$
\end{theorem}
\begin{proof}
Consider an arbitrary permutation $\alpha$ of 
  $\{1,\ldots,n\}$ in the class ${\mathcal C}$ having $k$ cycles and let
$\alpha'$ be the unique reduced permutation obtained by a sequence of
elementary reductions. Assume $\alpha'$ is a permutation of
$\{1,\ldots,n_1\}$ and has $k_1$ cycles. Thus $\mu(\alpha')$ is a
matching of $2n_1$ elements, these elements form a circular list
$(-1,1,\ldots,-n_1,n_1)$. Any removable edge of $\mu[\alpha]$ has to be
reinserted between two consecutive elements of this circular list. We
will adapt the results on pointing and substitution, described in the
book of Flajolet and Sedgewick~\cite[Section I.6.2]{Flajolet-Sedgewick},
to a two variable setting, keeping in mind the two-coloring of the
points. We will use the notation
$$
[x^ny^k] f(x,y) 
$$
to denote the coefficient of $x^ny^k$ in the formal power series
$f(x,y)\in {\mathbb R}[[x,y]]$.  

We distinguish two cases, and describe the
generating function of the permutations belonging to each case. We begin
with the more straightforward case.  

\bigskip
\noindent{\bf Case 1} {\em Neither $(1,-2)$ nor $(-1,1)$ is a removable
  edge of $\mu[\alpha]$.} In this case the
edge $(1,-\alpha'(1))$ of $\mu[\alpha']$ is the image of the edge 
$(1,-\alpha(1))$ of $\mu[\alpha]$ at the end of the removal
process. This leaves the label of $1$ unchanged, and we will have to
relabel the points from here with $1,-2,2,\ldots,n,-1$ in the circular order.
This observation determines the assignment of labels completely, we only
need to keep track of the sets of reinserted edges. 
For each of the arcs $(-1,1)$, $(1,-2)$, \ldots, $(-n_1,n_1)$,
$(n_1,-1)$ created by the points of $\mu[\alpha']$  we may reinsert a
set of pairwise noncrossing edges 
independently, the only restriction being that we want to reinsert
$n-n_1$ edges altogether, making sure that $\alpha$ has $k$ cycles. As
we will see below, the way to count the 
additional cycles created is different for the arcs $(-i,i)$ from that
of the arcs $(i,-(i+1))$. We have $n_1$ arcs of each type, and at the
level of generating functions we will have $2n_1$ factors, $n_1$ of each
type.  

Consider first all edges reinserted between $i$ and $-(i+1)$ (for some
$i>0$). These form a set of pairwise non-crossing edges such  that their
points listed in the circular order begin with a negative element and
end with a positive element, and the absolute value of the 
labels keeps (weakly) increasing as we parse the elements in the cyclic
order. The 
number of cycles created by the reinsertion of these edges is the number
of edges $(u,v)$ such that $u$ is negative and precedes $v$ in the
circular order. Indeed, a removable edge does not cross any edge, as
we parse the signed points in cyclic order, starting from $-1$, each
removable edge edge $(u,-v)$ with $0<u<v$ begins or continues a cycle,
and each removable edge $(u,-v)$ with $0<v<u$ completes a cycle. The
cycles completed by a removable edge consists only of removable edges,
these are the new cycles, contributed by the reinsertion of such
edges. On the other hand, if we remove
all other edges and keep only the removable edges inserted between $i$ and
$-(i+1)$, by Lemma~\ref{lem:ncbic}, this set of bicolored edges encodes
a noncrossing partition with the same number of parts as the number of
newly added cycles, after decreasing the absolute values of
all labels appropriately. Therefore, if we insert $n'\geq 0$ edges between
$i$ and $-(i+1)$ then the  number of ways to create $k'$ new cycles is
$[x^{n'}y^{k'}] D(x,y)$. Note that this includes the possibility of
$n'=k'=0$, that is, we may choose not to insert any edge between $i$ and
$-(i+1)$ at all. 

Consider next all edges reinserted between $-i$ and $i$ (for some
$i>0$). These form a set of pairwise non-crossing edges, such that the
points listed in the circular order begin with a positive element and
end with a negative element. As before, the number of cycles created by
the reinsertion of these edges is the 
number of edges $(u,v)$ such that $u$ is negative and precedes $v$ in the
circular order. On the other hand, if we remove all other edges and keep
only the ones inserted between $i$ and $-i$, to apply
Lemma~\ref{lem:ncbic} we need to swap the signs to make sure that the
point with the label of least absolute value, first parsed in the
cyclic order has negative sign. Thus inserting $n''\geq 0$ 
edges between $-i$ and $i$ then the number of ways to create $k''$ new
cycles can be done in as many ways as one can create a noncrossing
partition on $n''$ points with $n''-k''$ parts. This number is
$[x^{n''}y^{n''-k''}] (D(x,y)$ if $n'',k''>0$ and it is $1$ when
$n''=k''=0$. In the case when $n'',k''>0$ we obtain that 
\begin{align*}
[x^{n''}y^{n''-k''}] D(x,y)&= [x^{n''}y^{n''-k''}] (D(x,y)-1)=
[x^{n''}y^{k''+1}](D(x,y)-1)\\ 
&=[x^{n''}y^{k''}]\frac{D(x,y)-1}{y}
\end{align*}
Here we used the fact that the number of noncrossing partitions of $n''$
elements into $n''-k''$ parts is the same as the number of noncrossing
partitions of $n''$ elements into $k''+1$ parts. Note that each nonempty
partition has at least one part, hence $D(x,y)-1$ is a multiple of
$y$. To summarize, the number of ways to insert $n''\geq 0$ 
edges between $-i$ and $i$ while creating $k''$ new cycles is
$$
[x^{n''}y^{k''}]\left(\frac{D(x,y)-1}{y}+1\right)
=[x^{n''}y^{k''}]\frac{D(x,y)+y-1}{y}.
$$
Combining the contribution of the edges inserted on the arcs
$(i,-(i+1))$ for some $i$ with those inserted on the arcs $(-i,i)$ for some
$i$, we obtain that the number of partitions counted in this case is 
$$
\sum_{n_1,k_1\geq 0}  r(n_1,k_1) [x^{n-n_1}y^{k-k_1}] 
\left(D(x,y)\cdot\frac{D(x,y)+y-1}{y}\right)^{n_1}$$
which is exactly
$$
[x^ny^k] R\left(\frac{xD(x,y)(D(x,y)+y-1)}{y},y\right).
$$

\bigskip
\noindent{\bf Case 2} {\em Either $(1,-2)$ or $(-1,1)$ is a removable
  edge of $\mu[\alpha]$.}  Note that $1$ can only be matched with one of
$-1$ and $-2$, so this case has two mutually exclusive subcases. 

If $(1,-2)$ is a removable edge of $\mu[\alpha]$ then this edge is among
the edges that are to be reinserted between $1$ and $-2$ of the
bicolored matching $\mu(\alpha')$ associated to the reduced permutation
$\alpha'$. If we reinsert $n'$ edges, thus creating $k'$ cycles then we
also need to select one of the positive endpoints to be the $1$ in
$\alpha$. This can be done in
$$
n'\cdot [x^{n'}y^{k'}] D(x,y) = [x^{n'}y^{k'}] \frac{x
  \frac{\partial}{\partial x} D(x,y)}{D(x,y)}
$$
ways. Counting the parts created by inserting edges on the other arcs
created by $\alpha'$ is completely analogous to the process described in
the previous case. Hence the total number of all permutations counted in
this case is  
$$
[x^ny^k]  R\left(\frac{xD(x,y)(D(x,y)+y-1)}{y},y\right)
\frac{x \frac{\partial}{\partial x} D(x,y)}{D(x,y)}.
$$
A completely analogous reasoning shows that in the case when $(-1,1)$ is
a removable edge of $\alpha$, the total number of permutations counted
is 
$$
[x^ny^k]  R\left(\frac{xD(x,y)(D(x,y)+y-1)}{y},y\right)
\frac{x \frac{\partial}{\partial x} D(x,y)}{D(x,y)+y-1}.
$$
\end{proof}
The second stage of proving Theorem~\ref{thm:fromreduced} contains some
algebraic manipulations.  
Observe first that $D(x,y)$ may be equivalently given by the quadratic equation
\begin{equation}
\label{eq:Deq}
x\cdot D(x,y)^2+(xy-1-x) D(x,y)+1=0.
\end{equation}
An equivalent form of this equation is
\begin{equation}
\label{eq:Deq2}
xD(x,y)(D(x,y)+y-1)=D(x,y)-1.
\end{equation}
Dividing both sides of \eqref{eq:Deq2} by $y$ yields 
$$
\frac{xD(x,y)(D(x,y)+y-1)}{y}=\frac{D(x,y)-1}{y}.
$$
Thus the equations in Theorem~\ref{thm:fromreduced} and Theorem~\ref{thm:red2} 
respectively contain the same substitution into the function $R(x,y)$.
To conclude the proof of Theorem~\ref{thm:fromreduced} it suffices to show
the following. 

\begin{lemma}
The function $D(x,y)$ satisfies 
$$
1+\frac{x \frac{\partial}{\partial x} D(x,y)}{D(x,y)}+\frac{x
\frac{\partial}{\partial x} D(x,y)}{D(x,y)+y-1}
=
\frac{1}{\sqrt{(x+xy-1)^2-4x^2y}}.
$$
\end{lemma}
\begin{proof}
Taking the partial derivative with respect to $x$ on both sides of
\eqref{eq:Deq} we obtain
$$
D(x,y)^2+2xD(x,y)\frac{\partial}{\partial x}
D(x,y)-(1-y)D(x,y)-(1+x-xy)\frac{\partial}{\partial x} D(x,y)=0.
$$
Using this equation we may express $\frac{\partial}{\partial x} D(x,y)$
as follows:
\begin{equation}
\label{eq:Dpart1}
\frac{\partial}{\partial x} D(x,y)=\frac{D(x,y)(D(x,y)+y-1)}{1+x-xy-2xD(x,y)}.
\end{equation}
This equation directly implies
\begin{equation}
\label{eq:dfactor}
1+\frac{x \frac{\partial}{\partial x} D(x,y)}{D(x,y)}+\frac{x
\frac{\partial}{\partial x} D(x,y)}{D(x,y)+y-1}
=
\frac{1}{1+x-xy-2xD(x,y)}.
\end{equation}
Finally a direct consequence of \eqref{eq:dxy} we have 
\begin{equation}
\label{eq:discr0}
1+x-xy-2xD(x,y)=\sqrt{(x+xy-1)^2-4x^2y}.
\end{equation}
Combining \eqref{eq:dfactor} and \eqref{eq:discr0} yields the stated equality.
\end{proof}

\section{Parallel edges and primitive partitions}
\label{sec:primitive}

From now on we restrict our attention to partitions of genus $g$. As a
consequence of Theorem~\ref{thm:fromreduced}, it suffices to count reduced
partitions of a fixed genus, the rest follows by substitution into the
formula given there.

Our next way to simplify is the elimination of {\em parallel edges}.

\begin{definition}
Given a reduced permutation $\alpha$ of $\{1,\ldots,n\}$ and a pair of numbers
$\{i,j\}\subseteq \{1,\ldots,n\}$ such that $\alpha(i)=j+1$ and
$\alpha(j)=i+1$, we say that the ordered pairs $(i,\alpha(i))$ and
$(j,\alpha(j))$ are {\em parallel edges}.    
\end{definition}
To avoid confusion with
$2$-cycles we will use the notation $i\rightarrow \alpha(i)$ and
$j\rightarrow \alpha(j)$.
Note that a reduced permutation has no fixed points or dual fixed
points, and so for any pair of parallel edges, the set of points
$\{i,\alpha(i),j,\alpha(j)\}$ has $4$ elements.   

Direct substitution into the definition yields the following. 
\begin{lemma}
For any reduced permutation $\alpha$ of $\{1,\ldots,n\}$, the ordered
directed edges  $i\rightarrow \alpha(i)$ and $j\rightarrow \alpha(j)$
are parallel if and only if $(i,j)$ is a $2$-cycle of $\alpha^{-1}\zeta_n$.   
\end{lemma}

In general, a permutation $\alpha$ may have a pair of parallel edges
whose points all belong to the same cycle of $\alpha$. This is not the
case for partitions.

\begin{proposition}
\label{prop:gamma}
Let $\alpha$ be a reduced partition on $\{1,\ldots,n\}$ and let
$\{i\rightarrow \alpha(i), j\rightarrow  \alpha(j)\}$ be a pair of
parallel edges. Then $i$ 
and $j$ belong to different cycles of $\alpha$ and the permutation
$\gamma_{i,j}[\alpha]$, given by
$$
\gamma_{i,j}[\alpha](k)=
\begin{cases}
i+1 & \mbox{if $k=i$;}\\
j+1 & \mbox{if $k=j$;}\\  
\alpha(k) & \mbox{if $k\notin \{i,j\}$}\\
\end{cases} 
$$
is also a partition. Furthermore, $\gamma_{i,j}[\alpha]$ has the same
genus as $\alpha$.  
\end{proposition}  
\begin{proof}
The points $i,i+1,j,j+1$ follow each other in this cyclic order, as
shown in Figure~\ref{fig:parallel}.  
\begin{figure}[h]
\begin{center}
\input{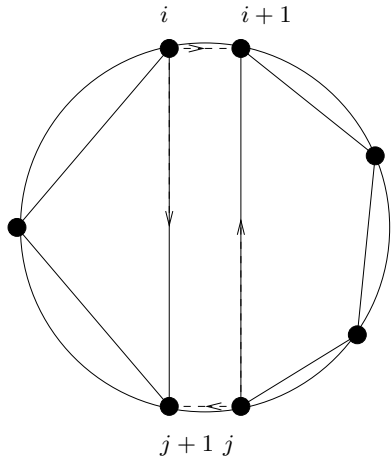}
\end{center}
\caption{A pair of parallel edges in a partition}
\label{fig:parallel}
\end{figure}

Since in each cycle of a permutation $\alpha$ representing a partition
there is a unique $k$ such that $\alpha(k) \leq k$, the sequence
$$ i,j+1, \cdots , j, i+1$$ cannot be a subsequence of a cycle of $\alpha$.
Removing the parallel edges $j\rightarrow i+1$ and $i
\rightarrow j+1$ and adding 
the directed edges $i \rightarrow i+1 $ and $j \rightarrow  j+1$ merges
the two polygons into 
a single polygon. The resulting permutation $\gamma_{i,j}[\alpha]$ is a
partition and we have $z(\gamma_{i,j}[\alpha])=z(\alpha)-1$. Note also
that $\gamma_{i,j}[\alpha]^{-1}\zeta_n$ is obtained from
$\alpha^{-1}\zeta_n$ by replacing the $2$-cycle $(i,j)$ with the
pair of fixed points $(i)(j)$. Hence
$z(\gamma_{i,j}[\alpha]^{-1}\zeta_n)=z(\alpha^{-1}\zeta_n)+1$, and the
genus is unchanged by (\ref{eq:permgenus}). 
\end{proof}
Note that $\gamma_{i,j}[\alpha]$ is not reduced, but we can make it
reduced by removing the dual fixed points $i$ and $j$.
\begin{definition}
Let  $\alpha$ be a reduced partition on $\{1,\ldots,n\}$ and let
$\{i \rightarrow \alpha(i)), j\rightarrow  \alpha(j)\}$ be a pair of
parallel edges. We will 
refer to taking $\gamma_{i,j}[\alpha]$ and then removing its dual fixed
points $i$ and $j$ using two elementary reductions of the second kind as
the {\em removal of the pair of parallel edges $\{i\rightarrow  \alpha(i),
  j \rightarrow \alpha(j)\}$}.    
\end{definition}  
A direct consequence of the definitions is the following.
\begin{lemma}
\label{lem:dp}  
The effect on $\alpha^{-1}\zeta_n$ of the removal of the pair of
parallel edges $\{i\rightarrow \alpha(i), j\rightarrow \alpha(j)\}$ is
the following. The $2$-cycle $(i,j)$ is deleted and each label $k$
in the remaining cycles is decreased by the number of elements in
$\{1,\ldots,k\}\cap \{i,j\}$.  
\end{lemma}

A special case of a removal of a pair of parallel edges is, when at
least one of these edges, say $i\rightarrow \alpha(i)$ is part of a
$2$-cycle. Merging this $2$-cycle with another polygon and then
removing the arising dual fixed points has the same pictorial effect as
simply removing this $2$-cycle. For example $\{1\rightarrow  5,
4\rightarrow  2\}$ is a
parallel pair of edges in Figure~\ref{fig:reduced}, and both are also
edges of $2$-cycles. The partition $\gamma_{1,4}[\alpha_1]$ contains
the ``rectangle'' $(1,2,4,5)$ in which $1$ and $4$ are dual fixed
points. The removal of these yields the partition $\alpha_2$ shown in
Figure~\ref{fig:semiprimitive}.  
\begin{figure}[h]
\begin{center}
\input{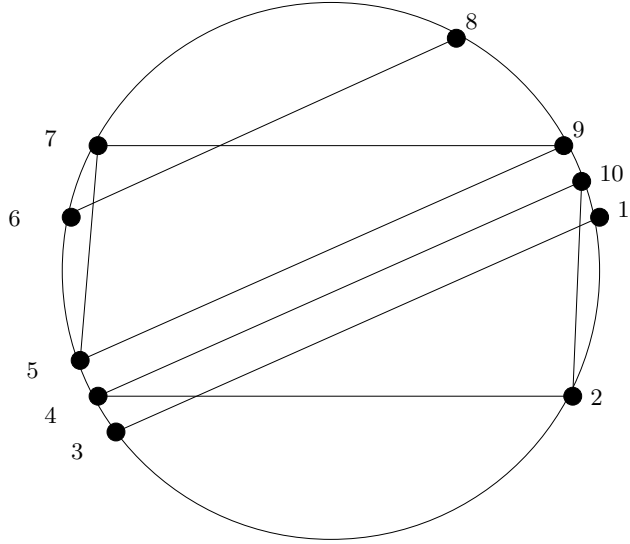}
\end{center}
\caption{The semiprimitive partition $\alpha_2$ associated to $\alpha_0$}
\label{fig:semiprimitive}
\end{figure}
\begin{definition}
We call a partition $\alpha$ {\em semiprimitive} if it has no pair of parallel
edges $\{i \rightarrow \alpha(i), j\rightarrow \alpha(j)\}$ such that
$(i,\alpha(i))$ is a 
$2$-cycle. We call a partition $\alpha$ {\em primitive} if it
contains no pairs of parallel edges at all. 
\end{definition}  
The partition $\alpha_2$ shown in Figure~\ref{fig:semiprimitive} is
semiprimitive: its only pair of parallel edges is $(\{9,5), (4,10)\}$
and both of these directed edges are parts of $3$-cycles. Removing this
last pair of parallel edges yields the primitive partition shown in
Figure~\ref{fig:primitive}.
\begin{figure}[h]
\begin{center}
\input{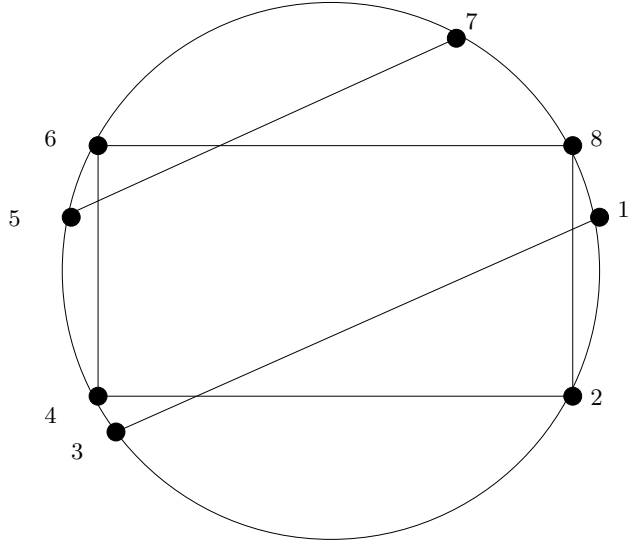}
\end{center}
\caption{The primitive partition $\alpha_3$ associated to $\alpha_0$}
\label{fig:primitive}
\end{figure}
\begin{proposition}
\label{prop:primitive}
By repeated removal of pairs of parallel edges, each reduced partition
 may be transformed into a primitive
partition.  This primitive partition does not depend on the order in which the removals
are performed. 
\end{proposition}  
\begin{proof}
This statement is easily shown by repeated use of Lemma~\ref{lem:dp}.
If $\alpha$ is a partition of $\{1,\ldots,n\}$ and $\alpha^{-1}\zeta_n$ contains $m$ $2$-cycles then $\pi[\alpha]$
is a partition on the set $\{1,\ldots,n-2m\}$. The cycle decomposition
of $\pi[\alpha]^{-1}\zeta_{n-2m}$ is obtained from the cycle decomposition of
$\alpha^{-1}\zeta_n$ by removing all $2$-cycles, and decreasing each
remaining label $k$ by the number of elements belonging to one of the
removed $2$-cycles and also to the set $\{1,\ldots,k\}$.
\end{proof}

In analogy to Proposition~\ref{prop:primitive}, we also have the
following statement. 

\begin{proposition}
\label{prop:semiprimitive}
By repeated removal of pairs of parallel edges such that at least one
edge in the pair is a $2$-cycle of $\alpha$, each reduced partition
$\alpha$ of $\{1,\ldots,n\}$ may be transformed into a unique semiprimitive
partition $\sigma[\alpha]$.   
\end{proposition}  
\begin{proof}
The proof depends on the following observation. The answer to the
question whether  $i$ and $\alpha(i)$ form a $2$-cycle
(equivalently: $i$ is a fixed point of $\alpha^2$) remains
essentially unchanged after the removal a pair of parallel edges, where
neither of the edges is $i\rightarrow \alpha(i)$. Only the labels $i$
and $\alpha(i)$ may  decrease accordingly. If initially
$\alpha^{-1}\zeta_n$ has $m'$ 
$2$-cycles $(i,j)$ such that $i$ or $j$   
is a fixed point of $\alpha^2$, then $\sigma[\alpha]$  is a partition of
$\{1,\ldots,n-2m'\}$. Furthermore $\sigma[\alpha]^{-1}\zeta_{n-2m'}$ is
obtained from $\alpha^{-1}\zeta_n$ by deleting all $2$-cycles
$(i,j)$ of $\alpha^{-1}\zeta_n$ such that $i$ or $j$ is a fixed point of
$\alpha^2$ and decreasing each label $k$ by the number of elements
removed from the set $\{1,\ldots,k\}$.  
\end{proof}  

The main result of this paper is a consequence of the following observation.

\begin{theorem}
\label{thm:finite}
A primitive partition of genus $g$ is a partition of a set with at most\\
$6(2g-1)$ elements. Moreover for any $g$ there is a finite  number of
semiprimitive partitions of genus $g$, hence also a finite number  of
primitive ones. 
\end{theorem}  
\begin{proof}
Let $\alpha$ be a primitive
partition of genus $g$ of the set $\{1,\ldots,n\}$. Since $\alpha$ is
reduced, it has no fixed points, every cycle of $\alpha$ has length at
least $2$ and $z(\alpha)\leq n/2$ . By the same reason,
$\alpha^{-1}\zeta_n$ has no fixed point either, and by the primitivity
of $\alpha$, there are no $2$-cycles in $\alpha^{-1}\zeta_n$
either, each of its cycles has length at least $3$. Thus
$z(\alpha^{-1}\zeta_n)\leq n/3$.  Equation
(\ref{eq:permgenus}), together with the  above observations, yields
$$
n+1-2g=z(\alpha)+z(\alpha^{-1}\zeta_n)\leq \frac{n}{2} + \frac{n}{3},
$$
and the stated inequality follows after rearranging.

Consider now a semiprimitive partition $\alpha$ of genus $g$. By
Proposition~\ref{prop:primitive}, after the removal of parallel pairs of
edges we arrive at the unique primitive partition $\pi(\alpha)$. By the
already shown part of the statement, $\pi(\alpha)$ is a partition of at
most $6(2g-1)$ elements. Since $\alpha$ is semiprimitive, each removal
of a parallel pair of edges results in merging two of its cycles, of length
$c_1\geq 3$ and $c_2\geq 3$ into a cycle of length $c_1+c_2$,  where the
removed parallel pair of edges corresponds to a diagonal of the cycle
that was created. For example, the primitive partition shown in
Figure~\ref{fig:primitive} arises by merging the two triangles of the
semiprimitive partition shown in
Figure~\ref{fig:semiprimitive}. Replaying the sequence of moves
resulting in $\pi(\alpha)$ backward, what we see in each step is that a
cycle of length at least $4$ of the current partition is cut into two smaller
polygons by cutting along a diagonal. For a fixed primitive partition,
this sequence of actions can be performed only in finitely many ways. 
\end{proof}  

\section{Counting reduced partitions}
\label{sec:primred}

What we have obtained so far is that each reduced partition  of
genus $g$ is obtained from one of finitely many semiprimitive partitions
by repeatedly adding $2$-cycles $(u,v)$ in such a way that
 either the directed edge $u\rightarrow v$ or the directed edge
$v\rightarrow  u$ forms a parallel pair with an existing edge of the
current partition. Such an operation does not change the number of parts
of size greater than $2$, the addition of $2$-cycles appears as adding
parallel line segments to the diagram of the partition. Each newly added
$2$-cycle is parallel to at least one already existing edge in the
diagram. In order to avoid ambiguities, let us have a closer look
whether it is possible that the lastly added $2$-cycle $(u,v)$ creates
more than one parallel pair of directed edges in the current partition
$\beta$ obtained from a semiprimitive partition $\alpha$ by repeatedly adding $2$-cycles .
This is only possible if we have 
$$\beta(u)=v\quad\mbox{and}\quad \beta(v-1)=u+1$$
as well as
$$\beta(v)=u\quad\mbox{and}\quad \beta(u-1)=v+1$$
implying that both $\{u\rightarrow  v, v-1 \rightarrow u+1\}$ and
$\{v\rightarrow u, u-1 \rightarrow v+1\}$ are
parallel pairs of directed edges in $\beta$. Note that in this case the
removal of either parallel pairs results in the parallel pair of
directed edges $\{u-1\rightarrow v, v-1\rightarrow  u\}$. In other
words, we have inserted a $2$-cycle in between a parallel pair of edges.
This inspires the following definition:
\begin{definition}
We define a {\em parallel class} of directed edges in a reduced
partition $\alpha$ as the reflexive and transitive closure of the
following relation: 
\begin{enumerate}
\item If $u\rightarrow \alpha(u)$ and $v\rightarrow \alpha(v)$ form a
  parallel pair of edges then they are in the same parallel class.  
\item For a $2$-cycle $(u,v)$ the directed edges $u\rightarrow v$
  and $v\rightarrow u$ are in the same parallel class. 
\end{enumerate}
\end{definition}
By a slight abuse of the terminology we will refer to $2$-cycles
being in the same parallel class, instead of saying that both directed edges of
the $2$-cycle are in the same parallel class.
An example of a reduced partition containing many directed edges in the
same parallel class is shown in Figure~\ref{fig:parallels}.
For this example, the parallel classes containing (directed edges of)
$2$-cycles are  
$$
\{(17,7),(18,6),(1,5),(2,4)\}, \quad \{11\rightarrow 16, (12,15)\},
\quad \mbox{and}\quad \{(9,14),(10,13)\}. 
$$
\begin{figure}[h]
\begin{center}
\input{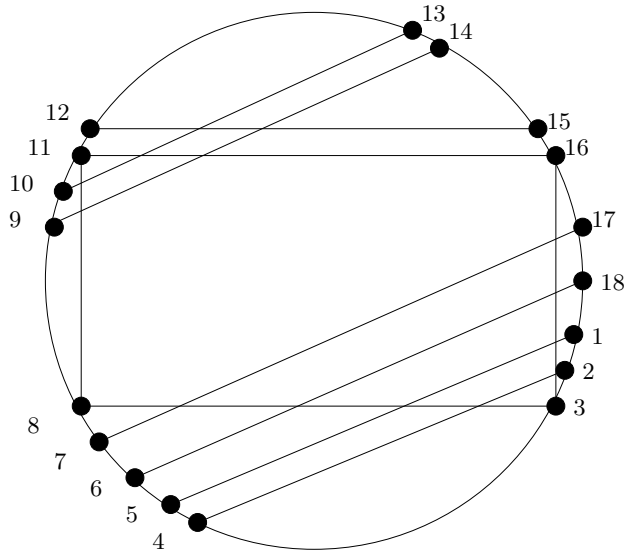}
\end{center}
\caption{Parallel $2$-cycles in a reduced permutation}
\label{fig:parallels}
\end{figure}
We may represent $2$-cycles in the same parallel class with parallel
line segments in the diagram of the partition. The converse does not
need to be true: the $2$-cycles $(9,14)$ and $(7,17)$ are
represented with parallel line segments, but they are not parallel. 
The parallel classes are still easily recognizable, because of the
following observation.
\begin{lemma}
\label{lem:consecutive}  
  In a reduced partition, the set of endpoints of all directed edges
  belonging to the same parallel class form two cyclically consecutive
sets of labels. The two sets can 
not be merged into a single cyclically consecutive set.
\end{lemma}
Indeed, between every two directed edges $u\rightarrow v$ and
$u'\rightarrow v'$ there is a sequence of edges such that any two
consecutive edges form either a parallel pair, or have the same pair of
endpoints. Using this observation it is easy to show the first statement
by induction on the number of edges. The second statement is an easy
consequence of the fact that there is no directed edge of the form
$u\rightarrow u+1$ in a reduced partition. 

The second key observation is the following.
\begin{lemma}
\label{lem:1stays}  
Let $\alpha$ be a partition, let $(u,v)$ be a $2$-cycle of $\alpha$ and
let $\{u\rightarrow v, u'\rightarrow v'\}$ be a parallel pair of
edges. The removal of this parallel pair decreases the number of
$2$-cycles in the parallel class of $u\rightarrow v$ by
one. Furthermore, the label $1$ will belong to a directed edge in the same
parallel class as before, and it will belong to the same consecutive set
of endpoints associated to that class.   
\end{lemma}  
The proof is left to the reader. 

As a consequence of Lemmas~\ref{lem:consecutive} and \ref{lem:1stays}, 
we may uniquely reconstruct a reduced partition $\alpha$ from the unique
semiprimitive partition $\sigma[\alpha]$ that can be obtained by the
removal of parallel pairs of directed edges from it, 
if we now the number of $2$-cycles added to each parallel
class of $\sigma[\alpha]$, and the location of the point labeled $1$ in
$\alpha$ (the latter is restricted by Lemma~\ref{lem:1stays}). 

The unique semiprimitive partition associated to the partition shown in
Figure~\ref{fig:parallels} is also primitive, it is the partition
$\alpha_3$ shown in Figure~\ref{fig:primitive}. The primitive partition
$\alpha_3$ allows the addition of $2$-cycles to $6$ parallel classes, each
corresponding to one edge of the diagram of $\alpha_3$. The direction of
this edge matters if it is part of a cycle longer than two: edges in the
parallel class of $\{(14,17)\}$ will appear all above the unique
$4$-cycle. On the other hand $2$-cycles that are in the
parallel class corresponding to a $2$-cycle of $\alpha_3$ may
have been added on either side of the original $2$-cycle.

For semiprimitive but not primitive partitions a minor complication
arises due to the fact that such a partition contains parallel pairs of
edges. For example, for the partition shown in
Figure~\ref{fig:semiprimitive} the first $2$-cycle inserted as a
parallel to $9\rightarrow 5$ is also parallel to $4\rightarrow
10$. Additional $2$-cycles in the same parallel class can not be
automatically associated to a single edge of $\alpha_2$. In such
situations we will make an arbitrary choice and mark one of the two
edges, say $9\rightarrow 5$ as the directed edge representing the
parallel class of $2$-cycles that may be merged into this edge by
repeated removals of parallel edges.

\begin{definition}
\label{def:pc}  
In a semiprimitive partition we select each edge of its diagram that is
not part of a parallel pair as a {\em parallel class representative} and
from each parallel pair of edges we select exactly one as a parallel
class representative. Subject to this selection we say that a point has
{\em type} $0$, $1$, or $2$, respectively if the number of edges
incident to it in the diagram that are parallel class representatives
is $0$, $1$, or $2$, respectively.
\end{definition}  
Subject to the selection of $9\rightarrow 5$ as a parallel class
representative, the type $1$ points of $\alpha_2$ are
$1,3,4,6,8,10$. These are the endpoints of the edges representing
$2$-cycles and the points $4$ and $10$ which are endpoints of the
edge $4\rightarrow 10$ that is not a parallel class representative.
The type $2$ points in the same diagram are $2,5,7,9$. There are no
type $0$ points. We will see in Section~\ref{sec:gen2} that no type $0$
points arise in genus $2$. This possibility may occur for higher
genuses if in a semiprimitive partition a point is contained in two
edges, both of which form a parallel pair of edges, and neither of them
is selected as a parallel class representative.

Subject to a last trick, we are now in the position to write a
generating function formula for all reduced partitions $\alpha$ for
which $\sigma[\alpha]$ is the same semiprimitive partition. 

\begin{definition}
Given a semiprimitive partition $\beta$, let us denote by
$r_{\beta}(n,k)$ the number of all reduced partitions $\alpha$ of
$\{1,\ldots,n\}$ into $k$ parts, satisfying $\sigma[\alpha]=\beta$. We
denote by $R_{\beta}(x,y)$ the generating function 
$R_{\beta}(x,y)= \sum_{n\geq 1, k\geq 1}r_{\beta}(n,k) x^ny^k $. 
\end{definition}  
Clearly $R(x,y)$, the generating function of all reduced partitions of
genus $g$, is the sum of $R_{\beta}(x,y)$ over all (finitely many)
semiprimitive partitions $\beta$ of genus $g$. The last trick we will
use is to compute {\em average contribution} of a semiprimitive
partition in a {\em cyclic recoloring class}.

Keeping in mind Proposition~\ref{prop:cyc}, together with each
semiprimitive partition $\beta$ on $\{1,\ldots,m\}$ we consider all
partitions of the form $\zeta_m^j\beta\zeta_m^{-j}$. These are all
partitions of the same genus, and they are also semiprimitive:

\begin{lemma}
Let $\beta$ be a partition of the set $\{1,\ldots,m\}$. If $\beta$ is
primitive or semiprimitive then the same holds for all
$\zeta_m^j\beta\zeta_m^{-j}$.  
\end{lemma}  
Indeed, cyclic relabeling does not change the cycle structure of $\beta$
or $\beta^{-1}\zeta_m$, and parallel pairs of directed edges are taken
into parallel pairs of directed edges under cyclic relabeling.

\begin{definition}
Let $\beta$ be a semiprimitive partition of $\{1,\ldots,m\}$ . We call
the {\em average contribution of $\beta$ to $R(x,y)$ modulo cyclic
  relabeling} the generating function
$$
R_{\overline{\beta}}(x,y)=
\frac{1}{m}\cdot 
\sum_{j=0}^{m-1} R_{\zeta_m^j\beta\zeta_m^{-j}}(x,y).
$$  
\end{definition}  
No matter how many elements are equivalent to $\beta$ modulo cyclic
relabeling, if we replace each $R_{\beta}(x,y)$ with
$R_{\overline{\beta}}(x,y)$, we over-count the contribution of each
equivalent semiprimitive partition $m$ times, and then we divide by
$m$. Therefore we have
\begin{equation}
  \label{eq:ravgb}
R(x,y)=\sum_{\beta}  R_{\overline{\beta}}(x,y)
\end{equation}  
where $\beta$ ranges over all semiprimitive partitions of genus $g$. 

\begin{theorem}
\label{thm:reduced}  
Let $\beta$ be a semiprimitive partition on $m$ points with $c$
cycles, whose diagram has $p$ parallel classes. Suppose we
have selected parallel class representatives as described in
Definition~\ref{def:pc} and, subject to this selection, $\beta$ has
$m_i$ points of type $i$ for $i=0,1,2$. Then the average contribution
of $\beta$ to $R(x,y)$ modulo cyclic relabeling is given by
\begin{align*}
R_{\overline{\beta}}(x,y)&=x^{m} y^{c}\cdot 
\frac{m_0\cdot (1-x^2y)+m_1+m_2\cdot(1+x^2y)}{m\cdot (1-x^2y)^{p+1}}\\
&= \frac{x^{m} y^{c}}{(1-x^2y)^{p+1}}  + 
\frac{(m_2-m_0)\cdot x^{m+2} y^{c+1}}{m\cdot (1-x^2y)^{p+1}} 
\end{align*}

\end{theorem}  
\begin{proof}
When we take the $m$ cyclically relabeled copies $\zeta^{j}_m\beta
\zeta_m^{-j}$ of $\beta$, we will keep the same directed edges as parallel
class representatives. This way the label $1$ will appear as a type $i$
point exactly $m_i$ times, where $i=0,1,2$ in these copies. In all
cases below the factor $x^my^c$ is there because each $\zeta^{j}_m\beta
\zeta_m^{-j}$ is a partition of $m$ elements with $c$ cycles.

{\bf\noindent Case 1}. If $1$ is a type $0$ point then to construct
the only choice we can make is to choose number of $2$-cycles added
to each parallel class of directed edges in $\zeta^{j}_m\beta
\zeta_m^{-j}$. Each parallel class contributes a factor of
$$
\frac{1}{1-x^2y}=\sum_{n=0}^{\infty} (x^2y)^n,
$$
as each $2$-cycle is a distinct part (contributing a factor of $y$)
with two points (factor of $x^2$).

{\bf\noindent Case 2}. If $1$ is a type $1$ point then, besides
selecting the number of $2$-cycles added to each parallel class, we
must also select the position of the label $1$ in $\alpha$. This
selection must be made within the parallel class whose representative
contains $1$, and the endpoint must be in the same consecutive set of
endpoints that contains the $1$ of $\zeta^{j}_m\beta\zeta_m^{-j}$. If
there are $n$ edges added in this parallel class, then this choice may
be performed $(n+1)$ ways. Hence the parallel class represented by the
directed edge containing the label $1$ in $\zeta^{j}_m\beta\zeta_m^{-j}$
contributes a factor of 
$$
\frac{1}{(1-x^2y)^2}=\sum_{n=0}^{\infty} (n+1)\cdot (x^2y)^n,
$$
all other parallel classes contribute a factor of $1/(1-x^2y)$. 

{\bf\noindent Case 3}. The case when $1$ is a type $2$ point is similar
to the previous case, except that now the $1$ of $\alpha$ may be the
endpoint of a directed edge in one of two parallel classes, represented
by a directed edge containing $1$ in  $\zeta^{j}_m\beta\zeta_m^{-j}$. If
these two classes contain $n$ added $2$-cycles then there are $n$
ways to select the number of $2$-cycles in one of the two classes
and $n$ ways to select the position of the label $1$ in $\alpha$. Hence
these two parallel classes contribute a factor of 
$$
\frac{1+x^2y}{(1-x^2y)^3}=\sum_{n=0}^{\infty} (n+1)^2\cdot (x^2y)^n,
$$
all other parallel classes contribute a factor of $1/(1-x^2y)$. 

\end{proof}

As a direct consequence of Theorem~\ref{thm:finite} and
Theorem~\ref{thm:reduced} we have the following result.

\begin{corollary}
\label{cor:reduced}
For a fixed genus $g$, the generating function $R(x,y)$ of reduced
partitions of genus $g$ is a rational function of $x$ and
$y$. Moreover, the denominator of $R(x,y)$ is a power of $1-x^2y$.
\end{corollary}  

Combining Corollary~\ref{cor:reduced} with Theorem~\ref{thm:fromreduced}
we obtain the main result of our paper.

\begin{theorem}
  \label{thm:main}
For a fixed $g$ the generating function
$P(x,y)=\sum_{n,k} p(n,k) x^ny^k$ of genus $g$ partitions of $n$
elements with $k$ parts is algebraic. More precisely, it may be obtained
by substituting $x$, $y$ and $\sqrt{(x+xy-1)^2-4x^2y}$ into a rational
expression.  
\end{theorem}

\section{Extracting the coefficients from our generating functions}
\label{sec:extract}

In this section, we describe a method to explicitly compute the
generating function $P(x,y)$ of all partitions of a fixed genus $g$ from
the generating function $R(x,y)$ of the reduced partitions. Our
procedure allows us to write $P(x,y)$ in such a form that a formula due
to Gessel may be used to extract the coefficient, thus obtaining
formulas for all such partitions of a set of a given size, with a given
number of parts. 

By Corollary~\ref{cor:reduced}, the generating function $R(x,y)$ of
reduced partitions is a linear combination of rational functions of the
form
\begin{equation}
\label{eq:rijk}  
r_{i_1,i_2,i_3}(x,y)=\frac{x^{i_1}y^{i_2}}{(1-x^2y)^{i_3}}.
\end{equation}
To obtain $P(x,y)$, by
Theorem~\ref{thm:fromreduced}, we need to replace each occurrence of $x$
by the quotient $(D(x,y)-1)/y$, and then multiply the result by the
multiplicative inverse of
\begin{equation}
  \label{eq:discr}
  \Delta(x,y)=\sqrt{(x+xy-1)^2-4x^2y}=1-x-xy-2x(D(x,y)-1).
\end{equation}
Thus, to obtain $P(x,y)$ from $R(x,y)$, we have to replace each 
$r_{i_1,i_2,i_3}(x,y)$ with 
\begin{equation}
\label{eq:pijk}    
p_{i_1,i_2,i_3}(x,y)
=\frac{\left(\frac{D(x,y)-1}{y}\right)^{i_1}
  y^{i_2}}{\left(1-\frac{(D(x,y)-1)^2}{y^2}\cdot y\right)^{i_3}}
\cdot
\Delta(x,y)^{-1}
=
\frac{\left(\frac{D(x,y)-1}{y}\right)^{i_1}
  y^{i_2}}{\left(1-\frac{(D(x,y)-1)^2}{y}\right)^{i_3}}\cdot
\Delta(x,y)^{-1}. 
\end{equation}
We may express each $p_{i_1,i_2,i_3}(x,y)$ in terms of $\Delta(x,y)$ as
follows. First we simplify the denominator of $p_{i_1,i_2,i_3}(x,y)$ using
the following lemma.
\begin{lemma}
\label{lem:denominator}
  We have
  $$
\frac{1}{1-\frac{(D(x,y)-1)^2}{y}}=\frac{1-(xy+x-1)\cdot \Delta(x,y)^{-1}}{2}.
  $$
\end{lemma}
\begin{proof}
After multiplying the numerator and the denominator by $\Delta(x,y)$, 
using (\ref{eq:discr}) and simplifying by $2$,  we may rewrite the
right hand side as 
$$
\frac{\Delta(x,y)-(xy+x-1)}{2\cdot\Delta(x,y)}=
\frac{1-x-xy-x(D(x,y)-1)}{1-x-xy-2x(D(x,y)-1)}.
$$
Multiplying the numerator and the denominator on the right hand side by
$xy$ the stated equality is equivalent to
\begin{equation}
\label{eq:Donly} 
\frac{xy}{xy-x(D(x,y)-1)^2}=
\frac{1-x-xy-x(D(x,y)-1)}{1-x-xy-2x(D(x,y)-1)}.
\end{equation}
Observe next that (\ref{eq:Deq}) may be rewritten as
$$
x(D((x,y)-1)^2+(xy+x-1)(D(x,y)-1)+xy=0. 
$$
Using this equation we may replace $-x(D((x,y)-1)^2$ with
$(xy+x-1)(D(x,y)-1)+xy$ on the left hand side of (\ref{eq:Donly}). We
obtain that the stated equality is equivalent to 
$$
\frac{xy}{2xy+(xy+x-1)(D(x,y)-1)}=
\frac{1-x-xy-x(D(x,y)-1)}{1-x-xy-2x(D(x,y)-1)}.
$$
The last equation is easily seen to be an equivalent form of
(\ref{eq:Deq}).  
\end{proof}  
Directly from the definition of $D(x,y)$ we obtain
\begin{equation}
\label{eq:Dy}  
\frac{D(x,y)-1}{y}=\frac{-(xy+x-1)-\Delta(x,y)}{2xy}.
\end{equation}  
Multiplying the equation stated in Lemma~\ref{lem:denominator} with
equation (\ref{eq:Donly}) we obtain the following result.
\begin{lemma}
  \label{lem:samedeg}
  We have
 $$
\frac{\frac{D(x,y)-1}{y}}{1-\frac{D(x,y)^2-1}{y}}=\frac{x}{\Delta(x,y)}.
 $$ 
\end{lemma}  
Indeed, the stated equality directly follows from
$$
(-(xy+x-1)-\Delta(x,y))\cdot \frac{\Delta(x,y)-(xy+x-1)}{\Delta(x,y)}
=
\frac{(xy+x-1)^2-\Delta(x,y)^2}{\Delta(x,y)}
$$
after simplification. 

\begin{proposition}
\label{prop:pijk}
  The expression $p_{i_1,i_2,i_3}(x,y)$ is given by
$$
  p_{i_1,i_2,i_3}(x,y)=
  \begin{cases}
\di \frac{x^{i_1}y^{i_2}}{\Delta(x,y)^{i_1+1}}    &\mbox{if $i_1=i_3$;}\\
\\
\di \frac{x^{i_1} y^{i_2}}{\Delta(x,y)^{i_1+1}}
\cdot\left(\frac{1-(xy+x-1)\cdot \Delta(x,y)^{-1}}{2}\right)^{i_3-i_1}
 &\mbox{if $i_1<i_3$;}\\
\\
\di \frac{x^{i_3} y^{i_2}}{\Delta(x,y)^{i_3+1}} \left(\frac{-(xy+x-1)-\Delta(x,y)}{2xy}\right)^{i_1-i_3} &\mbox{if $i_1>i_3$.}\\
  \end{cases}  
$$
\end{proposition}  
\begin{proof}
The first line is a direct consequence of Lemma~\ref{lem:samedeg}. To
obtain the second line, we write
$$
\frac{\left(\frac{D(x,y)-1}{y}\right)^{i_1}
  y^{i_2}}{\left(1-\frac{(D(x,y)-1)^2}{y}\right)^{i_3}}
=
\left(\frac{\frac{D(x,y)-1}{y}
}{1-\frac{(D(x,y)-1)^2}{y}}\right)^{i_1}\cdot y^{i_2} \cdot 
\frac{1}{\left(1-\frac{(D(x,y)-1)^2}{y}\right)^{i_3-i_1}},
$$
apply Lemma~\ref{lem:samedeg} to the first factor and
Lemma~\ref{lem:denominator} to the third factor. Finally, to obtain the
third line, we write
$$
\frac{\left(\frac{D(x,y)-1}{y}\right)^{i_1}
  y^{i_2}}{\left(1-\frac{(D(x,y)-1)^2}{y}\right)^{i_3}}
=
\left(\frac{\frac{D(x,y)-1}{y}
}{1-\frac{(D(x,y)-1)^2}{y}}\right)^{i_3}\cdot y^{i_2} \cdot 
\left(\frac{D(x,y)-1}{y}\right)^{i_1-i_3},
$$
we apply Lemma~\ref{lem:samedeg} to the first factor and
equation~(\ref{eq:Dy}) to the third factor. 
\end{proof}  
\begin{corollary}
The generating function $P(x,y)$ is a linear combination of Laurent
polynomials of $x$ and $y$, multiplied with negative powers of
$\triangle(x,y)$. 
\end{corollary}
Indeed, note that any even positive power of $\Delta(x,y)$ is a polynomial, and
any odd positive power of of $\Delta(x,y)$ may be written as a product
of a polynomial and of $\Delta(x,y)^{-1}$.

It remains to show how to compute the coefficient of $x^ny^k$ in a
negative integer power of $\Delta(x,y)$. For this purpose, Gessel's
following formula may be used, see~\cite[Eq.~(2)]{Gessel}:
$$
\frac{1}{(1-2x-2y+(x-y)^2)^{\alpha}}=\sum_{i,j\geq 0}
\frac{(\alpha+1/2)_{i+j}(2\alpha)_{i+j}}{i!j!
  (\alpha+1/2)_{i}(\alpha+1/2)_{j}}  x^{i} y^{j}.
$$
Here each $(u)_m=u(u+1)\cdots (u+m-1)$ is a rising factorial. 
As pointed out by Strehl~\cite[p.~180]{Strehl} (see also \cite[p.~64]{Gessel}),
\cite[Eq.~(2)]{Gessel} is a consequence of classical results in the
theory of special functions. We used this formula in~\cite{Cori-Hetyei}
in the special case when $\alpha$ is the half of an odd integer, to
count partitions of genus $1$. As noted in~\cite{Cori-Hetyei},
replacing each appearance of $y$ with $xy$ yields a formula
for a negative power of $\Delta(x,y)$:
$$
\frac{1}{(1-2x(1+y)+x^2(1-y)^2)^{\alpha}}=\sum_{i,j\geq 0}
\frac{(\alpha+1/2)_{i+j}(2\alpha)_{i+j}}{i!j!
  (\alpha+1/2)_{i}(\alpha+1/2)_{j}}  x^{i+j} y^{j}.
$$
Next we replace $j$ with $k$ and $i$ with $n-k$. Thus we obtain:
\begin{equation}
\label{eq:Gessel}  
\frac{1}{(1-2x(1+y)+x^2(1-y)^2)^{\alpha}}=\sum_{n\geq k\geq 0}
\frac{(\alpha+1/2)_{n}(2\alpha)_{n}}{(n-k)!k!
  (\alpha+1/2)_{n-k}(\alpha+1/2)_{k}}  x^{n} y^{k}.
\end{equation}
In the case when $\alpha=m+1/2$ is the half of an odd integer, it is
easy to see that
\begin{equation}
\label{eq:Gessel2}  
\frac{(\alpha+1/2)_{n}(2\alpha)_{n}}{(n-k)!k!
  (\alpha+1/2)_{n-k}(\alpha+1/2)_{k}}
=
\frac{\binom{n+2m}{m}\binom{n+m}{k}\binom{n+m}{n-k}}{\binom{2m}{m}}\quad\mbox{holds}.
\end{equation}
When we count partitions of a higher fixed
genus, substituting integer values of $\alpha$ may also be necessary.

\section{Partitions of genus one}
\label{sec:gen1}

As a ``warm up'', in this section we apply the results of the preceding
sections and quickly reproduce the generating function formula for genus
one partitions, first found in~\cite{Cori-Hetyei}.  
For these partitions Theorem~\ref{thm:finite} gives that any
primitive partition is a partition on at most $6(2\cdot 1-1)=6$
elements. Simple trial and error gives that the only primitive
partitions are $\beta_1=(1,3)(2,4)$ and $\beta_2=(1,4)(2,5)(3,6)$. Since
all cycles have length $2$, these are also the only semiprimitive
partitions and all points have type $1$. 
As a consequence of Theorem~\ref{thm:reduced}, the generating function
of all reduced partitions of genus one is
$$
R(x,y)=R_{\overline{\beta_1}}(x,y)+R_{\overline{\beta_2}}(x,y)=
x^{4} y^{2}\cdot 
\frac{4}{4\cdot (1-x^2y)^{3}}
+
x^{6} y^{3}\cdot 
\frac{6}{6\cdot (1-x^2y)^{4}}.
$$
\begin{corollary}
\label{cor:pg1red}
The generating function $R(x,y)$ of reduced partitions of genus one is 
$$
R(x,y)=\frac{x^4y^2}{(1-x^2y)^4}=r_{4,2,4}(x,y).
$$
\end{corollary}
Corollary~\ref{cor:pg1red}, combined with the first line of
Proposition~\ref{prop:pijk}, immediately yields the following result,
first shown in~\cite{Cori-Hetyei}. 

\begin{theorem}
\label{thm:pnkgf}
The generating function $P(x,y)$ of all partitions of genus one is given
by 
$$
P(x,y)=\frac{x^4y^2}{(1-2(1+y)x+x^2(1-y)^2)^{5/2}}.
$$
\end{theorem}

\section{Genus $2$ partitions}
\label{sec:gen2}

For genus $2$ partitions Theorem~\ref{thm:finite} gives that any
primitive partition is a partition on at most $6(2\cdot 2-1)=18$
elements. First we have a closer look at the possible cycle lengths of
these partitions.

\begin{theorem}
For a primitive partition $\beta$ of genus $2$ on $n$ elements, one of the
following holds:
\begin{enumerate}
\item $\beta$ has only $2$-cycles and $n\leq 18$;  
\item $\beta$ has one $3$-cycle, all other cycles are $2$-cycles and
  $n\leq 15$;
\item $\beta$ has two $3$-cycles, all other cycles are $2$-cycles and
  $n\leq 12$;
\item $\beta$ has one $4$-cycle, all other cycles are $2$-cycles and
  $n\leq 12$.
\end{enumerate} 
\end{theorem}  
\begin{proof}
Observe that for $g=2$ Equation (\ref{eq:permgenus}) gives
\begin{equation}
\label{eq:pg2}
n-3=z(\beta)+z(\beta^{-1}\zeta_n).  
\end{equation}
Note also that the primitivity of $\beta$ implies
\begin{equation}
\label{eq:zdual}  
z(\beta^{-1}\zeta_n)\leq \frac{n}{3}
\end{equation}

Assume first, by way of
contradiction, that $\beta$ has at least three cycles whose length is at
least $3$. Let $c_1,c_2$ and $c_3$ be the length of three such cycles. 
In this case
$$z(\beta)\leq 2+ \frac{n-c_1-c_2-c_3}{2}.$$
Combining this with (\ref{eq:pg2}) and (\ref{eq:zdual}) we obtain 
$$
n-3\leq 2+ \frac{n-c_1-c_2-c_3}{2} + \frac{n}{3}, \quad\mbox{that is,}
$$
$$
n\leq 30-3(c_1+c_2+c_3).
$$
Since $c_1+c_2+c_3\geq 9$, the above inequality yields $n\leq 3$, in
contradiction with $n\geq c_1+c_2+c_3\geq 9$. Therefore there are at most
two cycles in $\beta$ that are not involutions. 

{\bf\noindent Case 1:} $\beta$ has  exactly two cycles that are not
$2$-cycles. Let the length of these cycles be $c_1$ and $c_2$. In this
case 
$$z(\beta)= 2+ \frac{n-c_1-c_2}{2}.$$
Combining this with (\ref{eq:pg2}) and (\ref{eq:zdual}) yields
$$
n-3\leq 2+ \frac{n-c_1-c_2}{2} + \frac{n}{3}, \quad\mbox{that is,}
$$
\begin{equation}
\label{eq:c1c2}  
  n\leq 30-3(c_1+c_2).
\end{equation}  
Assume, by way of contradiction, that at least one of $c_1$ and $c_2$ is
greater than $3$. Then $c_1+c_2\geq 7$, and above inequality gives
$n\leq 9$. In that case, $\beta^{-1}\zeta_n$ has at most $3$ cycles, but
it can not have exactly $3$, as three $3$-cycles have at most $3$ back
points, in contradiction with Corollary~\ref{cor:backp} which requires at
least $4$ back points. Thus we must have $z(\beta^{-1}\zeta_n)\leq 2$
and  (\ref{eq:pg2}) yields 
$$
n-3\leq 2+ \frac{n-c_1-c_2}{2}+2, \quad\mbox{implying}\quad
n\leq 14-(c_1+c_2)\leq 7.
$$
Now, if $z(\beta^{-1}\zeta_n)=2$ then $\beta^{-1}\zeta_n$ has one three
cycle and one cycle of length at most four, with at most $3$ back points,
which is impossible. Thus we must have $z(\beta^{-1}\zeta_n)=1$ and
(\ref{eq:pg2}) yields 
$$
n-3\leq 2+ \frac{n-c_1-c_2}{2}+1, \quad\mbox{implying}\quad
n\leq 14-(c_1+c_2)\leq 5.
$$
This contradicts $n\geq c_1+c_2\geq 7$. We obtained that in this 
case $\beta$ has  exactly $2$ cycles of length $3$ and (\ref{eq:c1c2})
yields $n\leq 12$. 

{\bf\noindent Case 2:} $\beta$ has exactly one cycle whose length is
greater than $2$. Let $c_1$ be the length of this cycle. In this case we have
$$
z(\beta)=1+\frac{n-c_1}{2},
$$
and $c_1$ has the same parity as $n$. 
Equation (\ref{eq:pg2})
yields
$$
n-3\leq 1+\frac{n-c_1}{2}+\frac{n}{3}\quad\mbox{implying}
$$
\begin{equation}
\label{eq:c1}  
n\leq 24-3c_1.
\end{equation}
Assume, by way of contradiction that $c_1\geq 5$ holds. In this case
(\ref{eq:c1}) yields $n\leq 9$. Just like in the previous case,
$\beta^{-1}\zeta_n$ can not have three $3$-cycles with altogether at
most $3$ back points, thus we must have $z(\beta^{-1}\zeta_n)\leq 2$. 
Equation (\ref{eq:pg2})
yields
$$
n-3\leq 1+\frac{n-c_1}{2}+2\quad\mbox{implying}
$$
$$
n\leq 12-c_1
$$
Since $c_1\geq 5$, we obtain $n\leq 7$. Just like in the previous case
$\beta^{-1}\zeta_n$ can not have one $3$-cycle and one cycle of length
at most $4$, as these could contain at most $3$ back points. We are
left with $z(\beta^{-1}\zeta_n)=1$ and  Equation (\ref{eq:pg2})
yields
$$
n-3\leq 1+\frac{n-c_1}{2}+1\quad\mbox{implying}
$$
$$
n\leq 10-c_1
$$
Since $c_1\geq 5$, we obtain $n\leq 5$. This, together with $n\geq c_1$
forces $n=c_1=5$ and $\beta=(12345)$, a partition of genus $0$, not $2$. 
This contradiction proves that we can only have $c_1=3$ or $c_1=4.$
For $c=3$ (\ref{eq:c1}) Equation gives $n\leq 15$, for $c=3$ Equation
(\ref{eq:c1}) gives $n\leq 12$. 

{\bf\noindent Case 3:} All cycles of $\beta$ are $2$-cycles. In this
case Theorem~\ref{thm:finite} implies $n\leq 18$.
\end{proof}   

We have found all primitive partitions using a computer search. They are
shown in Table~\ref{tab:primitives}. The last line indicates the number
of parts as a function of $n$.

\begin{table}[h]
\begin{tabular}{r||c|c|c|c|}
$n$ & Transpositions only & one $3$-cycle & two $3$-cycles & one $4$-cycle\\ 
  \hline
  \hline
  $6$  & $0$   & $0$   & $1$  & $0$\\
  $7$  & $0$   & $14$  & $0$  & $0$\\
  $8$  & $21$  & $0$   & $20$ & $6$\\
  $9$  & $0$   & $141$ & $0$  & $0$\\
  $10$ & $168$ & $0$   & $65$ & $15$\\
  $11$ & $0$   & $407$ & $0$  & $0$\\
  $12$ & $483$ & $0$   & $52$ & $9$\\
  $13$ & $0$   & $455$ & $0$  & $0$\\
  $14$ & $651$ & $0$   & $0$  & $0$\\
  $15$ & $0$   & $0$   & $0$  & $0$\\
  $16$ & $420$ & $0$   & $0$  & $0$\\
  $17$ & $0$   & $0$   & $0$  & $0$\\
  $18$ & $105$ & $0$   & $0$  & $0$\\   
\end{tabular} 
\caption{Numbers of primitive partitions of genus $2$ }
\label{tab:primitives}  
\end{table}  

\subsection{Reduced matchings of genus $2$}

If a primitive partition contains only $2$-cycles then all points have
type $1$ and the number of parallel classes is the same as the number of
$2$-cycles, that is $n/2$. By Theorem~\ref{thm:reduced} the generating
function of all reduced partitions associated to such primitive
partitions is
\begin{align*}
R_M(x,y)&=21 \cdot x^8y^4 \cdot \frac{1}{(1-x^2y)^5}
 +168 \cdot x^{10}y^5    \cdot \frac{1}{(1-x^2y)^6}
 +483 \cdot x^{12}y^{6}  \cdot \frac{1}{(1-x^2y)^7}\\
 &+651 \cdot x^{14}y^7    \cdot \frac{1}{(1-x^2y)^8}
 +420 \cdot x^{16}y^8    \cdot \frac{1}{(1-x^2y)^9}
 +105 \cdot x^{18}y^9    \cdot \frac{1}{(1-x^2y)^{10}}.\\
\end{align*}
This expression may be simplified to

\begin{equation}
\label{eq:redm}
R_M(x,y)=21\cdot (x^2y)^4\cdot \frac{1+3\cdot x^2y+(x^2y)^2}{(1-x^2y)^{10}}.  
\end{equation}  
Note that this is the generating function of all {\em reduced matchings
  of genus $2$}, that is, the generating function of all reduced
partitions of genus $2$ in which all parts have size $2$. All partitions
of genus $2$ with this property have already been counted and this gives
us a means to verify the numbers we found with computer in this case.

In analogy to Theorem~\ref{thm:fromreduced}, it is not hard to show
that the generating function $\sum_{n\geq 0} m_2(n)\cdot t^n$ for all
matchings of genus $2$ with $n$ edges is 
$R_M(C(t),t)$, where
$$
C(t)=\frac{1-\sqrt{1-4t}}{2t}
$$
is a generating function of the Catalan numbers (the coefficient of
$t^n$ counts the number of noncrossing partitions of $2n$ such that each
part has two elements). Evaluating the Taylor series of $R_M(C(t),t)$
gives
$$
R_M(C(t),t)= t^4+483\cdot t^5+6468\cdot t^6+66066\cdot t^7+570570\cdot
t^8+4390386\cdot t^9+31039008\cdot t^{10}+\cdots  
$$
The coefficients are listed as sequence A006298 in \cite{OEIS}, as
counting ``genus $2$ rooted maps with $1$ face with n points'', the
main reference being the work of Walsh and Lehman~\cite{Walsh-Lehman}.  

\subsection{The contribution of the other primitive partitions}

If a primitive partition of genus $2$ of $\{1,\ldots,n\}$ contains only
one $3$-cycle then $3$ points have type $2$ and $n-3$ points have
type $1$. The number of parallel classes is $3+(n-3)/2=(n+3)/2$. 
By Theorem~\ref{thm:reduced} the generating
function of all reduced partitions associated to such primitive
partitions is
\begin{align*}
R_{\triangle} (x,y)&=
14\cdot x^7y^3\cdot \frac{4+3\cdot (1+x^2y)}{7\cdot (1-x^2y)^6}+
141\cdot x^9y^4\cdot \frac{6+3\cdot (1+x^2y)}{9\cdot (1-x^2y)^7}\\
&+407\cdot x^{11}y^5\cdot \frac{8+3\cdot (1+x^2y)}{11\cdot (1-x^2y)^8}+
455\cdot x^{13}y^6\cdot \frac{10+3\cdot (1+x^2y)}{13\cdot (1-x^2y)^9}.
\end{align*}
This expression may be simplified to
\begin{equation}
\label{eq:onet}  
R_{\triangle} (x,y)=\frac{7\cdot x^7y^3(2+13 x^2y+ 13
  (x^2y)^2+2(x^2y)^3)}{(1-x^2y)^{10}}. 
\end{equation}

If a primitive partition of genus $2$ of $\{1,\ldots,n\}$ contains two
$3$-cycles then $6$ points have type $2$ and $n-6$ points have
type $1$. The number of parallel classes is $6+(n-6)/2=(n+6)/2$. 
By Theorem~\ref{thm:reduced} the generating
function of all reduced partitions associated to such primitive
partitions is
\begin{align*}
R_{\davidsstar}(x,y)&=
x^6y^2\cdot \frac{6(1+x^2y)}{6(1-x^2y)^7}
+20\cdot x^8y^3\cdot \frac{2+6(1+x^2y)}{8(1-x^2y)^8}\\
&+65\cdot x^{10}y^4\cdot \frac{4+6(1+x^2y)}{10(1-x^2y)^9}
+52\cdot x^{12}y^5\cdot \frac{6+6(1+x^2y)}{12(1-x^2y)^{10}}
\end{align*}
This expression may be simplified to
\begin{equation}
\label{eq:twot}  
R_{\davidsstar} (x,y)=\frac{x^6y^2(1+18 x^2y+55
  (x^2y)^2+30(x^2y)^3+(x^2y)^4)}{(1-x^2y)^{10}}. 
\end{equation}

Finally, if a primitive partition of genus $2$ of $\{1,\ldots,n\}$
contains one $4$-cycle then $4$ points have type $2$ and $n-4$ points have
type $1$. The number of parallel classes is $4+(n-4)/2=(n+8)/2$. 
By Theorem~\ref{thm:reduced} the generating
function of all reduced partitions associated to such primitive
partitions is
$$
R_{\Box}(x,y)=
6\cdot  x^8y^3 \cdot\frac{4+4(1+x^2y)}{8(1-x^2y)^7}
+15\cdot  x^{10}y^4 \cdot\frac{6+4(1+x^2y)}{10(1-x^2y)^8}
+9\cdot  x^{12}y^4 \cdot\frac{8+4(1+x^2y)}{12(1-x^2y)^9}.
$$
This expression may be simplified to 
\begin{equation}
\label{eq:oneb}  
R_{\Box} (x,y)=\frac{6 x^8y^3(1+x^2y)}{(1-x^2y)^{9}}.
\end{equation}

\subsection{Semiprimitive partitions that are not primitive}

A semiprimitive partition that is not primitive has at least one
parallel pair of directed edges whose removal leads to a cycle of length
at least $4$. Continuing the removal of parallel pairs of directed
edges, the primitive partition reached must also have a cycle of length
at least $4$. In genus $2$ this is only possible when the primitive
partition has a single $4$-cycle, and each semiprimitive partition
arises from cutting the unique $4$-cycle into two $3$-cycles. An example
of such a semiprimitive partition is shown in
Figure~\ref{fig:semiprimitive}, removing its pair of parallel edges
results in the primitive partition shown in Figure~\ref{fig:primitive}.   

When counting semiprimitive partitions, it is important to the number of
equivalent partitions modulo cyclic relabeling is usually different for
the semiprimitive partition and the associated primitive partition.
For example, the primitive partition shown in Figure~\ref{fig:primitive}
has $4$ distinct partitions in its equivalence class
($\zeta_8^4\alpha_2\zeta_8^{-4}=\alpha_2$), whereas the semiprimitive
partition shown in Figure~\ref{fig:semiprimitive} has $5$ distinct
partitions in its equivalence class
($\zeta_{10}^5\alpha_3\zeta_{10}^{-5}=\alpha_3$). Writing a program that
searches for all semiprimitive partitions of genus $2$ that are not
primitive is more complicated then identifying primitive partitions,
luckily the semiprimitive partitions can also be identified ``by
hand''. We spare the reader the details of the tedious work, we
summarize our findings in Table~\ref{tab:semiprimitive}. 

\begin{table}[h]
\begin{center}
\begin{tabular}{|l|c|l|c|l|}
\hline
$n$ & Primitive partition & cyclic & Semiprimitive
  partition & cyclic\\
& & copies & & copies\\
  \hline
  \hline
$10$ & $(1,3,5,7) (2,8) (4,6)$ & $4$ & $(1,3,5)(6,8,10)(2,9)(4,7)$ &
  $5$\\
     &                         &     & $(1,3,9)(4,6,8)(2,10)(5,7)$ &
  $5$\\
     & $(1,3,5,7) (2,6) (4,8)$ & $2$ & $(1,3,5)(6,8,10)(2,7)(4,9)$ &
  $5$\\
  \hline
$12$ & $(1,4,7,9) (2,5) (3,6)(8,10)$ & $10$ &
  $(1,4,7)(8,10,12)(2,5)(3,6)(9,11)$ & 
  $12$\\  
    &                                &     &
  $(1,4,11)(5,8,10)(2,6)(3,7)(9,12)$ & 
  $12$\\
    & $(1,4,6,9) (2,7) (3,8)(5,10)$ & $5$ &
  $(1,4,6)(7,10,12)(2,8)(3,9)(5,11)$ & 
  $6$\\
     &                               &    &
  $(1,3,6)(7,9,12)(2,8)(4,10)(5,11)$ & 
  $6$\\
\hline  
$14$ & $(1,4,7,10) (2,5) (3,6)(8,11) (9,12)$ & $6$ &
  $(1,4,7)(8,11,14)(2,5)(3,6)(9,12)(10,14)$ & 
  $7$\\
     &  &  &
  $(1,4,12)(5,8,11)(2,6)(3,7)(9,13)(10,14)$ & 
  $7$\\
     & $(1,4,7,10) (2,8) (3,9)(5,11) (6,12)$ & $3$ &
  $(1,4,7)(8,11,14)(2,9)(3,10)(5,12) (6,13)$ & 
$7$\\
\hline
\end{tabular}
\end{center}
\caption{Semiprimitive partitions of genus $2$}
\label{tab:semiprimitive}
\end{table}  

The number $n$
refers to the size of the underlying set of the semiprimitive
partition. The underlying set of the associated primitive partition has
two less elements. We listed one representative from each cyclic relabeling
equivalence class, and indicated the number of equivalent
(semi-)primitive partitions up to cyclic relabelings. On the left hand
side we see that there are $4+2=6$ primitive partitions containing a
$4$-cycle on $8$ elements, $10+5=15$ on $10$ elements and $6+3=9$ on
$12$ elements. In most cases we get $2$ inequivalent semiprimitive
partitions by cutting the $4$-cycle of a primitive partition along one
of its diagonals, in the case when only one semiprimitive partition is
indicated, cutting along the other diagonal yields an equivalent
partition up to cyclic relabeling. By adding up the numbers associated
to the same value of $n$ in the last column, we see that there are $15$
semiprimitive (but not primitive) partitions on $10$ elements, $36$ such
partitions on $12$ elements and $21$ such partitions on $14$ elements.

For such a semiprimitive partition on $n$ elements there are $4$
points of type $2$ and $n-4$ points of type $1$. The number of
parallel classes is $5+(n-6)/2=(n+4)/2$. By Theorem~\ref{thm:reduced}
the generating function of all reduced partitions associated to such
semiprimitive partitions is
\begin{align*}
R_{\boxbslash}(x,y)=&
15\cdot x^{10}y^4 \cdot \frac{6+4(1+x^2y)}{10(1-x^2y)^8}
+36\cdot x^{12}y^5 \cdot \frac{8+4(1+x^2y)}{12(1-x^2y)^9}\\
&+21\cdot x^{14}y^6 \cdot \frac{10+4(1+x^2y)}{14(1-x^2y)^{10}}
\end{align*}
This expression may be simplified to
\begin{equation}
  \label{eq:sp}
R_{\boxbslash}(x,y)=3 x^{10} y^4\cdot \frac{5+4x^2y}{(1-x^2y)^{10}}.  
\end{equation}  

\subsection{Adding up the contributions}

Adding up the generating functions given in (\ref{eq:redm}),
(\ref{eq:onet}), (\ref{eq:twot}), (\ref{eq:oneb}) and (\ref{eq:sp}) we
obtain the following result.

\begin{theorem}
\label{thm:reducedg2}
The generating function $R(x,y)=\sum_{n,k} r(n,k) x^ny^k$ of the numbers
$r(n,k)$ of all reduced genus $2$ partitions of $n$ elements with $k$
parts is given by 
$$
R(x,y)=\frac{x^6y^2\cdot r(x,y)}{(1-x^2 y)^{10}}
$$
Here
\begin{align*}
  r(x,y)&= 1+14 xy+24 x^2 y+21 x^2y^2+91 x^3 y^2\\
  &+55 x^4 y^2+63 x^4y^3\\
  &+91 x^5 y^3+21x^6 y^4+24 x^6 y^3+14 x^7 y^4+x^8 y^4.\\ 
\end{align*}  
\end{theorem}  

Using the method described in Section~\ref{sec:extract}, we computed a
closed form formula for the generating function $P(x,y)$ of all genus
$2$ partitions. It seems possible to perform this calculation by hand,
but the use of a computer algebra system can greatly reduce this lengthy
procedure. We relied on the help of Maple, and obtained the following
formula for the generating function $P(x,y)=\sum_{n,k} p(n,k)\cdot
x^ny^k$ of the numbers $p(n,k)$ of all genus $2$ partitions of $n$
elements with $k$: 

\begin{align*}
P(x,y) &=
\frac{x^{10}}{8}\cdot \frac{57 y^6-40 y^5- 90 y^4+72 y^3+ y^2}{\Delta(x,y)^{11}}
+\frac{x^9}{2}
\cdot \frac{-22 y^5+ 131 y^4-30 y^3-y^2}{\Delta(x,y)^{11}}\\
&+\frac{3x^8}{4}\cdot (y^4+4 y^3+y^2)\cdot 
\left(\frac{1}{\Delta(x,y)^9}+\frac{1}{\Delta(x,y)^{11}}\right)\\
&+\frac{x^7}{2}\cdot (6 y^3- y^2)\cdot
\left(\frac{3}{\Delta(x,y)^9}+\frac{1}{\Delta(x,y)^{11}}\right)\\
&+ \frac{x^6y^2}{8}\cdot
\left(\frac{1}{\Delta(x,y)^7} +\frac{6}{\Delta(x,y)^9}
+\frac{1}{\Delta(x,y)^{11}}\right) 
\end{align*}  
Here $\Delta(x,y)$ is the algebraic expression given in (\ref{eq:discr}).
Using the fact that $\Delta(x,y)^2=(-1+x+xy)^2-4x^2y$ we may replace
$1/\Delta(x,y)^9$ with $((-1+x+xy)^2-4x^2y)/\Delta(x,y)^{11}$ and
$1/\Delta(x,y)^7$ with $((-1+x+xy)^2-4x^2y)^2/\Delta(x,y)^{11}$ in the
formula above. Thus we obtain the following result.

\begin{theorem}
\label{thm:genus2gf}
The generating function $P(x,y)=\sum_{n,k} p(n,k)\cdot  x^ny^k$ of the numbers
$p(n,k)$ of all genus $2$ partitions of $n$ elements with $k$
parts is given by
$$
P(x,y)=\frac{x^6y^2\cdot p(x,y)}{(x^2-2x^2y-2x+x^2y^2-2xy+1)^{11/2}} 
$$
Here 
\begin{align*}
  p(x,y)&=x^4\cdot (8y^4-4y^3-15y^2+10y+1)+
          x^3\cdot (-4y^3+39y^2-10y-4)\\
          &+x^2\cdot (-15y^2-10y+6)
          +x\cdot (10y-4)+1
\end{align*}
\end{theorem}  

Using (\ref{eq:Gessel2}) it is easy to extract the coefficient of
$x^ny^k$ from the formula stated in Theorem~\ref{thm:genus2gf}.
\begin{theorem}
\label{thm:pnk}  
The number $p(n,k)$ of genus $2$ partitions of $n$ elements with $k$
parts is given by 
\begin{align*}
  p(n,k)&=
  8\cdot\gamma(n-10, k-6)-4\cdot\gamma(n-10, k-5)-15\cdot\gamma(n-10, k-4)\\
  &+10\cdot\gamma(n-10, k-3)+\gamma(n-10, k-2)\\
  &-4\cdot\gamma(n-9, k-5)+39\cdot \gamma(n-9, k-4)-10\cdot\gamma(n-9,
  k-3)-4\cdot\gamma(n-9, k-2)\\
  &-15\cdot\gamma(n-8, k-4)-10\cdot\gamma(n-8, k-3)+6\cdot\gamma(n-8,
  k-2)\\
  &-4\cdot\gamma(n-7, k-2)+10\cdot\gamma(n-7, k-3)+\gamma(n-6, k-2)
\end{align*}  
Here
$$
\gamma(n,k)=\frac{\binom{n+10}{5}\binom{n+5}{k}\binom{n+5}{n-k}}{\binom{10}{5}}.
$$
\end{theorem}
Using Maple, it is also possible to define $P(x,y)$ directly, by 
combining Theorem~\ref{thm:reducedg2} with
Theorem~\ref{thm:fromreduced}, and then use the computer algebra system
to directly compute its Taylor series. We had Maple compute the Taylor
series two ways: first by defining $P(x,y)$ by combining
Theorem~\ref{thm:reducedg2} with Theorem~\ref{thm:fromreduced} and then
by using the formula stated in Theorem~\ref{thm:genus2gf}. We obtained
the same coefficients, which also agree with the numbers obtained using
the formula  given in Theorem~\ref{thm:pnk}. 
The resulting numbers of partitions of genus $2$
may be found in Table~\ref{tab:pg2}, up to $n=12$.
 
\begin{table}[h]
\begin{center}
\begin{tabular}{|r||r|r|r|r|r|r|r|}
\hline
\backslashbox{$n$}{$k$} & $2$ & $3$ & $4$& $5$ & $6$ & $7$ &8\\  
\hline
\hline
$6$ & $1$ &&&&&&\\
$7$  &$7$ & $21$ &&&&&\\
$8$  &$28$ & $210$ & $161$&&&&\\
$9$  &$84$ & $1134$ & $2184$& $777$ &&&\\
$10$ &$210$ & $4410$ & $15330$& $13713$ & $2835$&&\\
$11$ &$462$ & $13860$ & $75075$& $121275$ & $63063$& $8547$&\\
$12$ &$924$ & $37422$ & $289905$ & $729960$ & $685608$ & $233772$ & $22407$\\
\hline
\end{tabular}
\end{center}
\caption{Numbers $p(n,k)$ of partitions of genus $2$ of $n$ elements with $k$ parts}
\label{tab:pg2}
\end{table}
These numbers agree with the numbers published in~\cite[Page, 48 Table 3.2]{Yip}
except for the values of $p(11,4)$ and $p(11,5)$, where
M. Yip's computer search has found $p(11,4)=75675$ and
$p(11,5)=110880$. Our exhaustive search for partitions of genus $2$ has
found the same numbers as published in our table.

\section*{Acknowledgments}
The second author wishes to express his heartfelt thanks to Labri, Universit\'e
Bordeaux I, for hosting him as a visiting researcher in Spring 2017,
when this research was started. This work was partially supported by a grant from the Simons Foundation
(\#245153 to G\'abor Hetyei).


\begin{thebibliography}{99}

\bibitem{CautisJackson}
S.\ Cautis and D. \ M. \ Jackson,
On Tutte's chromatic invariants,
{\it Trans.\ Amer.\ Math.\ Soc.\ }{\bf 362} (2010), 509--535.

\bibitem{Chapuy}
G.\ Chapuy,
The structure of unicellular maps, and a connection between maps of positive
genus and planar labelled trees, 
{\it Probab.\ Theory Relat.\ Fields} {\bf 147} (2010), 415--447.

\bibitem{Cori}
R. Cori
``Un code pour les Graphes Planaires et ses applications''
{\it Asterisque} {\bf 27} (1975).

\bibitem{Cori-Hetyei}
  R.\ Cori and G.\ Hetyei,
  Counting genus one partitions and permutations,
  {\it S\'em.\ Lothar.\ Combin.\ } {\bf 70} (2013) Art.\ B70e, 29 pp. 

\bibitem{Flajolet-Sedgewick}
P.\  Flajolet and R.\ Sedgewick, 
Analytic combinatorics, Cambridge University Press, Cambridge, 2009.
  
\bibitem{Gessel}
I.\ M.\ Gessel, 
On the number of convex polyominoes, 
{\it Ann.\ Sci.\ Math.\ Qu\'ebec} {\bf 24} (2000), 63--66. 
  
  

  
\bibitem{Goulden-Jackson}
I.\ P.\ Goulden and D.\ M.\ Jackson, 
Maps in locally orientable surfaces, the double coset algebra, and zonal
polynomials, 
{\it Canad.\ J.\ Math.\ } {\bf 48} (1996), 569--584.   


\bibitem{Goupil-Schaeffer}
A.\ Goupil and G.\ Schaeffer, 
Factoring $N$-cycles and Counting Maps of Given Genus,
{\it European J.\ Combin.\ } {\bf 19 } (1998), 819--834.

 
\bibitem{Hetyei-L}
G.\ Hetyei, 
Delannoy orthants of Legendre polytopes,
{\it Discrete Comput.\ Geom.} {\bf 42} (2009), 705--721. 

\bibitem{Jackson-Visentin1}
D.\ M.\ Jackson and T.\ I.\ Visentin, 
A character-theoretic approach to embeddings of rooted maps in an
orientable surface of given genus, 
{\it Trans.\ Amer.\ Math.\ Soc.\ } {\bf 322} (1990), 343--363.   

\bibitem{Jackson-Visentin2}
D.\ M.\ Jackson and T.\ I.\ Visentin, 
Character theory and rooted maps in an orientable surface of given genus: face-colored maps.
{\it Trans.\ Amer.\ Math.\ Soc.\ } {\bf 322} (1990), 365--376.    

\bibitem{Jacques}
A.\ Jacques,
Sur le genre d'une paire de substitutions,
{\it C.\  R.\ Acad.\ Sci.\ Paris} {\bf 267} (1968), 625--627.
  
\bibitem{Kreweras} 
G.\ Kreweras, 
Sur les partitions non crois\'ees d'un cycle,
{\it Discrete Math.} {\bf 1} (1972), 333--350.

  
\bibitem{OEIS}
N.J.A.\ Sloane,
``On-Line Encyclopedia of Integer Sequences,''\\
{\tt http://www.research.att.com/\~{}njas/sequences}

\bibitem{Simion}
R.\ Simion, 
A type-$B$ associahedron,
{\it Adv.\ Appl.\ Math.} {\bf 30} (2003), 2--25. 

\bibitem{Strehl}
V.\ Strehl, 
Zykel--Enumeration bei lokal-strukturierten Funktionen, 
Habilitationsschrift, Institut
f\"ur Mathematische Maschinen und Datenverarbeitung der Universit\"at
Erlangen--N\"urnberg, 1990.


\bibitem{Walsh-Lehman}
T.\ R.\ S.\ Walsh and A.\ B.\  Lehman, 
Counting rooted maps by genus. I,
{\it J.\ Combinatorial Theory Ser.\ B} {\bf 13} (1972), 192--218.   

\bibitem{Yip}
M.\ Yip, 
Genus one partitions, 
master's thesis, University of Waterloo,
2006; available online at 
{\tt http://hdl.handle.net/10012/2933}


\end{thebibliography}
\end{document}